\documentclass[10pt,a4paper]{amsart}

\usepackage{amsmath,amssymb}
\usepackage{a4wide}  
\usepackage{mathrsfs}
\usepackage{enumerate}
\usepackage{latexsym,amsfonts}
\usepackage{eucal}
\usepackage{verbatim}
\usepackage{color}
\usepackage{pb-diagram}
\usepackage{fancyheadings}
\title{Minimum-time strong optimality of a singular arc: the
multi-input non involutive case
}
\author{Francesca Chittaro}
\address{Aix Marseille Universit\'e, CNRS, ENSAM, LSIS UMR 7296, 13397 Marseille, France\\ 
and 
Universit\'e de Toulon, CNRS, LSIS UMR 7296, 83957 La Garde, France, {\tt\small francesca-carlotta.chittaro@univ-tln.fr}}
\author{Gianna Stefani}
\address{DIMAI, via S. Marta 3 - 50137 Firenze, Italy, {\tt\small gianna.stefani@unifi.it}.}
\thanks{This work was  supported by  Digiteo grant {\it Congeo}; by the ANR project {\it GCM}, program ``Blanche'',
project number NT09\_504490; by the European Research Council, ERC
StG 2009 ``GeCoMethods", contract number 239748; by PRIN 200894484E\_002 ; and by the research fund CARTT, IUT Toulon--La Garde.}
\date{}

\newtheorem{lemma}{Lemma}
\newtheorem{proposition}{Proposition}
\newtheorem{theorem}{Theorem}
\newtheorem{remark}{Remark}
\newtheorem{definition}{Definition}

\newtheorem{ass}{Assumption}

\numberwithin{theorem}{section}
\numberwithin{lemma}{section} 
 \numberwithin{proposition}{section} 
 \numberwithin{remark}{section} 
 \numberwithin{cor}{section} 
 \numberwithin{equation}{section} 
 \numberwithin{definition}{section} 

\newcommand{\wh}{\widehat}

\newcommand{\wq}{\widehat{q}}
\newcommand{\wxi}{\widehat{\xi}}
\newcommand{\wT}{\widehat{T}}
\newcommand{\wuu}{\widehat{\mathbf{u}}}
\newcommand{\wu}{\widehat{\mathbf{u}}}

\newcommand{\U}{\mathfrak{U}}
\newcommand{\vv}{\mathbf{v}}
\newcommand{\wla}{\widehat{\lambda}}
\newcommand{\well}{\widehat{\ell}}

\newcommand{\R}{\mathbb{R}}

\newcommand{\esse}{\mathcal{S}}
\newcommand{\bt}{\boldsymbol{t}}
\newcommand{\bzero}{\boldsymbol{0}}

\newcommand{\bth}{\boldsymbol{\vartheta}}
\newcommand{\bff}{\boldsymbol{\vec{F}}}
\newcommand{\J}{\boldsymbol{J}^{\prime \prime}}

\newcommand{\tpsi}{\psi}

\newcommand{\ba}{\boldsymbol{a}}
\newcommand{\bb}{\boldsymbol{b}}
\newcommand{\wbnu}{\widetilde{\boldsymbol{\nu}}}
\newcommand{\bnu}{\boldsymbol{\nu}}

\newcommand{\Lie}[1]{\mathrm{Lie}(#1)}

\begin{document}

\maketitle

\begin{abstract}
We consider the minimum-time problem for a
multi-input control-affine system, where we assume that the controlled vector fields generate a non-involutive distribution of constant dimension, and where we do not
assume a-priori bounds for the controls.
We use
Hamiltonian methods to prove that the coercivity of a suitable
second variation associated to a Pontryagin singular arc is sufficient to
prove its strong-local optimality. We provide an application of the result to a generalization of Dubins problem.
\end{abstract}

\section{Introduction}

In this paper we are concerned with the minimum-time problem associated with a control-affine system with several controls: 
\begin{equation} \label{eq: min T}
\min T 
\end{equation}
\noindent
subject to
\begin{equation} \label{eq: contr sys}
\left\{
\begin{array}{l}
\dot{\xi} = \left( f_0 + \sum_{i=1}^m u_i f_i \right) \circ \xi(t) \\
\xi(0) \in N_0, \quad \xi(T) \in N_f\\
\mathbf{u}=(u_1,\ldots,u_m) \in U \subset \mathbb{R}^m
\end{array} \right. .
\end{equation}
\noindent
The state $q$ belongs to a smooth $n$-dimensional manifold $M$, $f_0,f_1,\ldots,f_m$, are smooth vector fields on $M$, 
$N_0$ and $N_f$ are smooth submanifolds of $M$ and the control
functions belong to $L^\infty([0,T],U)$. We remark that for smooth we 
mean $C^{\infty}$.

We are interested in sufficient conditions for the strong-local optimality of \emph{singular} Pontryagin extremals of problem 
\eqref{eq: min T}-\eqref{eq: contr sys}, were \emph{strong} means \emph{with respect to the $C^0$-norm} of the trajectories $\xi(\cdot)$, 
and \emph{singular} means that $\mathbf{u}\in \mathrm{int}\, U$.
More precisely, we assume that there exists a candidate Pontryagin extremal $\wla : [0,\wT] \rightarrow T^*M$ with associated control function $\wu(\cdot) \in 
L^{\infty}([0,\wT],\mathrm{int}\, U)$
that satisfies $\pi\wla(0)\in N_0$ and $\pi\wla(\wT)\in N_f$, and we look for sufficient conditions that guarantee the strong-local optimality of the 
trajectory $\wxi=\pi\wla$, according to the following definition:
\begin{definition}
  The trajectory $\widehat\xi$ is a strong--local minimizer of the above considered problem,
  if there exist a neighborhood $V$ of its graph in $\mathbb{R} \times M$ and $\epsilon >0$ such that
  $\widehat\xi$ is a minimizer among the admissible trajectories whose
  graph is contained in $V$ and whose final time is greater than $\widehat T -
  \epsilon$, independently on the values of the associated controls.
\end{definition}
\noindent
This notion has been called
{\em time-state--local optimality} in \cite{PoggStebsb}, where also a stronger version of optimality
is considered.

The only assumption we do on the control set $U$ is that it has non-empty interior; although by Filippov's Theorem (\cite{AgSac}) 
we know that the existence of the minimum is
guaranteed when $U$ is compact and convex, here the existence of a candidate minimizer is taken as assumption.

A classical approach to sufficient optimality conditions is to consider the second variation 
(see for instance \cite{ASZ98,AgSac,bonz,Dmitruk77,GabKir72,SteZez97} and references therein). 
In particular, in \cite{ASZ98} and \cite{AgSac} the authors propose the definitions of an \emph{intrinsically defined} second variation, which is invariant
for coordinate changes, and therefore suitable to study optimal control problems defined on smooth manifolds.
  
A peculiarity of control-affine minimum-time problems is that the second variation does not contain the Legendre term (that is, the term which is quadratic 
in the variations of the control), thus turning out to be singular. A tool widely used to overcome this problem is the so-called 
Goh transformation \cite{Goh66}.
Thanks to this transformation,  performed in a coordinate--free way, we are able to convert the second variation proposed in \cite{ASZ98} into another functional 
which is no more singular, and thus it can be asked to be coercive
with respect to the $L^2$-norm of the new control variable. This approach, both
for classical and intrinsically-defined second variations, has been 
widely used in the analysis of sufficient optimality conditions (see for instance \cite{bonz,Dmitruk08,PoggStebsb,SteRoma}, and references therein).
                                                                                                                                                                  
In optimal-control problems, a classical method to prove the optimality of a Pontryagin extremal is to cover a neighborhood of the reference trajectory with
other admissible trajectories, to lift them to the cotangent bundle, and compare the costs evaluated along each trajectory.
In the standard theory, the trajectories to be lifted are obtained by projecting suitable solutions of the Hamiltonian system associated with	 the
maximized Hamiltonian $F_{max}$, see for example \cite{ASZ98,AgSac}.
This \emph{Hamiltonian method} is particularly effective, since it allows us to compare trajectories that belong to a $C^0$-neighborhood of the reference 
trajectory, independently on the value of the control.

When the extremal is singular $F_{max}$ cannot be used (see \cite{PoggStebsb}), then to construct the lifted trajectories we consider the solutions
of a system governed by a Hamiltonian greater than or equal to $F_{max}$,
as suggested by the approach used in \cite{PoggStebsb,SteRoma}.

Ultimately, the paradigm to get sufficient optimality condition for singular extremals combines an approach based on the coercivity of the second variation with the 
Hamiltonian approach. 
It relies on the following facts.
\begin{itemize}
 \item Under some regularity conditions, it is possible to define a smooth super-Hamiltonian whose flow is tangent to all singular extremals.
\item The derivative of the  super-Hamiltonian flow
is, up to an isomorphism, the Hamiltonian flow associated to the linear-quadratic problem
given by the second variation.
 \item If the second variation is coercive, it is possible  to transform the linear-quadratic problem associated with the original one into 
 a problem with free initial point, whose 
second variation is still coercive. In particular, this implies that the space of initial constraints for the linear-quadratic problem remains horizontal (that is,
it projects bijectively on $M$) under the action of the associated Hamiltonian flow.
\item The previous points imply  that the projection on $M$ of the super-Hamiltonian
flow emanating from the Lagrangian manifold associated with the initial conditions of the new
problem is locally invertible. As a result we get that it is possible to lift the trajectories
to the cotangent bundle, in order to apply the Hamiltonian method.
\end{itemize}

In the single-input case, problem \eqref{eq: min T}-\eqref{eq: contr sys} has already been studied in \cite{PoggStebsb}, where it has been shown that the 
coercivity of the second variation is a sufficient condition for the strong-local optimality of singular Pontryagin extremals.
In \cite{ChiSte} the authors studied the multi-input problem under the assumptions that the controlled vector fields
generate an involutive distribution. 
In this paper we relax this condition, that is we allow the controlled vector field to generate a non-involutive distribution.

We remark that our result remains true even if $U=\mathbb{R}^m$, then we need stronger assumptions than the usual ones. 
In particular, we have to consider High Order Goh condition (Assumption \ref{ass GGoh}), which we prove to be indeed a necessary optimality condition when 
$U=\mathbb{R}^m$. 
This phenomenon is not pointed out when the Lie algebra generated by the controlled vector 
fields is involutive, in particular when the system is single-input.
Indeed, in these cases High Order Goh condition is automatically satisfied under Goh condition.

We believe that this result, applied to the case where $U$ is an unbounded set, could be of help in the study of the infimum-time problem where the ``optimal'' 
trajectories may contain jumps, as in \cite{bressan-rampazzo-comm,bressan-rampazzo-noncomm,sar-guerra}, where integral costs are considered.

The structure of the paper is the following: we state the regularity assumptions in Section \ref{not e ass}; in Section \ref{sec secvar} we define the second variation and 
investigate the implications of its coercivity; the Hamiltonian method is exposed in Sections \ref{sec geometry} and \ref{sec results}, where we state and prove 
the main result; in Section \ref{sec example} we provide an example, based on a high dimensional version of Dubins' problem.
In the Appendices there are technical details on some results stated in the paper.  

\section{Notations and regularity assumptions} \label{not e ass}

In this section we clarify the notation we will use throughout the paper, and we state the regularity assumptions on the system.

Let $f$ be a vector field on the manifold $M$ and $\varphi : M \to \mathbb{R}$ be a smooth function. The action of $f$ on $\varphi$ (directional derivative or 
Lie derivative) evaluated on a point $q$ is denoted with the two expressions
\[
L_f \varphi(q)= \langle d\varphi(q),f(q)\rangle.
\]

The Lie brackets of two vector fields $f,g$ are denoted as commonly with $[f,g]$.
When dealing with vector fields labeled by indexes, we will use the following notations to denote their Lie brackets:
\[
f_{ij}(q)= [f_i,f_j](q),\qquad f_{ijk}(q)= [f_i,[f_j,f_k]](q).
\]

We call $\mathfrak{f}$ the set of the controlled vector fields of the control system \eqref{eq: contr sys}, that is $\mathfrak{f}=\{f_1,\ldots,f_m\}$, 
and Lie($\mathfrak{f}$) the Lie algebra generated by the set $\mathfrak{f}$. 
We denote Lie$_q(\mathfrak{f})=\mathrm{span}\{f(q) : f\in \mathrm{Lie}(\mathfrak{f})\}$.
In the following, for every $q\in M$, we call $\mathcal{I}_{q}$ the integral manifold of the distribution
$\Lie{\mathfrak{f}}$ passing through $q$.
The first assumption of this paper concerns the regularity of Lie($\mathfrak{f}$).

\begin{ass}   \label{ass lie}
The controlled vector field $f_1,\ldots,f_m$ are linearly independent 
and the Lie algebra $\Lie{\mathfrak{f}}$  has constant dimension $R$.
\end{ass}

Let us consider the cotangent bundle $T^*M$ of $M$, and let $\pi$ denote the canonical projection on $M$. It is well known that $T^*M$ possesses a canonically defined 
symplectic structure, given by the symplectic form 
$\sigma_{\ell}=d\varsigma(\ell)$, where $\ell$ denotes an element of $T^*M$ and $\varsigma$ is the Liouville canonic 1-form $\varsigma(\ell)=\ell\circ \pi_*$. 

We denote with the corresponding capital letter the Hamiltonian function associated with every vector field on $M$, that is 
$F(\ell)=\langle \ell,f(\pi \ell) \rangle$.

\begin{remark}
Let us recall that the following relation between the Lie brackets of two vector fields $f,g$ and the Poisson brackets of their associated Hamiltonian functions holds:
\[
\langle \ell, [f,g](\pi \ell) \rangle=\{F,G\}(\ell).
\] 
\end{remark}
\noindent
As above, we denote 
\[
F_{ij}(\ell)=\{F_i,F_j\}(\ell) \qquad  F_{ijk}(\ell)=\{F_i,\{F_j,F_k\}\}(\ell).
\]

We recall that the symplectic structure allows us to associate, with each Hamiltonian function $F$, the Hamiltonian vector field $\vec{F}$ on $T^*M$
defined by the action
\[
\langle d F(\ell), \cdot \rangle = \sigma_{\ell}(\cdot,\vec{F}(\ell)).
\]

In the following, we consider some special Hamiltonian functions associated with the optimal control problem \eqref{eq: min T}-\eqref{eq: contr sys}: 
 the (time-dependent) \emph{reference Hamiltonian}
\begin{equation} \label{ref Ham}
\widehat{F}_t(\ell)=F_0(\ell) + \sum_{i=1}^m \widehat{u}_i(t) F_i(\ell),
\end{equation}
where $\widehat{\mathbf{u}}(\cdot)$ is the \emph{reference control},
\noindent
and the maximized Hamiltonian
\[
F_{max} (\ell)=\sup_{\mathbf{u}\in U} \Big(F_0(\ell) + \sum_{i=1}^m u_i F_i(\ell)\Big).
\]

The Hamiltonian flow from time 0 to time $t$ associated with the reference Hamiltonian, that is the solution of the equation 
$\dot{\ell}=\vec{\widehat{F}}_t(\ell)$,
is denoted with $\widehat{\mathcal{F}}_t$. 

\bigskip

Let us consider an \emph{admissible triple} $(\widehat{\xi},\widehat{\mathbf{u}},\widehat{T})$ for the problem \eqref{eq: contr sys}, 
that is a solution of the control system; 
let us assume that $\widehat{\mathbf{u}} \in \mathrm{int} \, U$, and let us set $\wq_0=\wxi(0)$ and $\widehat{q}_f=\wxi(\wT)$.
We study the strong-local optimality of the triple $(\widehat{\xi},\widehat{\mathbf{u}},\widehat{T})$, that in the following we call \emph{reference triple},
among all solutions of \eqref{eq: contr sys} with
$N_0\subset \mathcal{I}_{\wq_0}$ and $N_f\subset \mathcal{I}_{\wq_f}$. 
In particular, Assumption \ref{ass lie} can be asked to hold only in a neighborhood of the reference trajectory. 

A classical necessary condition for the optimality of the reference triple $(\widehat{\xi},\widehat{\mathbf{u}},\widehat{T})$ is the
Pontryagin Maximum Principle (PMP), that we recall here stated in its Hamiltonian form (see \cite{AgSac}).
PMP states that if a reference trajectory $(\widehat{\xi},\widehat{\mathbf{u}},\widehat{T})$ satisfying $\widehat{\mathbf{u}} \in \mathrm{int} \, U$ 
is time-optimal, then there exist
a Lipschitzian curve $\wla:[0,\wT] \rightarrow T^*M$ and $p_0\in \{0,1\}$ that satisfy the following equations:
\begin{align}
\wla(t)&\neq 0\qquad \forall \ t\in [0,\wT] \label{eq: nonzero}\\
\pi \wla(t)&=\wxi(t) \qquad \forall \ t\in [0,\wT]\\
\frac{d}{dt}\wla(t)&=\vec{\widehat{F}_t}(\wla(t)) \qquad \forall \ t\in [0,\wT]  \label{PMP Ham}  \\
F_i(\wla(t))&=0 \qquad \forall \ i=1,\ldots,m \quad \forall \ t\in [0,\wT] \label{eq Fi=0} \\
\widehat{F}_t(\wla(t))&=F_0(\wla(t))=p_0 \qquad \forall \ t\in [0,\wT] \label{eq p0} \\
\wla(0)|_{T_{\wq_0}N_0}=0 & \qquad \wla(\wT)|_{ T_{\wq_f}N_f }=0.\label{eq: trans cond}
\end{align}
The Lipschitzian curves that satisfy equations \eqref{eq: nonzero}--\eqref{eq: trans cond}
 are called \emph{extremals}.
If $p_0=1$ we 
say the the extremal $\wla$ is \emph{normal}, while in the other case we say that it is \emph{abnormal}.

\begin{ass} We assume that the reference triple satisfies the PMP in the normal form, and we call 
the extremal $\wla$ \emph{reference extremal}.
\end{ass}
\noindent

By differentiating with respect to time, we obtain the following condition:
\begin{equation} \label{eq F0i+Fij}
F_{0i}(\wla (t)) + \sum_{j=1}^m \widehat{u}_j(t) F_{ji}(\wla (t))=0 \qquad i=1,\ldots,m,\quad  \mathrm{a.e.} \ t \in [0,\wT].
\end{equation}

In literature additional necessary conditions for the optimality of a singular extremal are known (see \cite{AgSac}). 
Namely, if the reference triple is optimal, then there exists an extremal $\lambda$ associated with the reference triple that satisfies 
the following conditions: 

\smallskip
\noindent
\textbf{(Goh condition)} 
\[
F_{ij}(\lambda(t))=0 \qquad \forall \ i,j=1,\ldots,m,\  t\in[0,\wT].
\]

\smallskip
\noindent
\textbf{(Generalized Legendre Condition)}  the quadratic form 
\begin{equation}
\mathbb{L}_{\lambda(t)}: \mathbf{v} \mapsto \sum_{i,j=1}^m \mathrm{v}_i \mathrm{v}_j F_{ij0} ({\lambda}(t)) + 
\sum_{i,j,k=1}^m \mathrm{v}_i \mathrm{v}_j \widehat{u}_k(t) F_{ijk}({\lambda}(t)) 
\label{eq: leg quad form} 
\end{equation}
is non-positive for any $\mathbf{v}=(\mathrm{v}_1,\ldots,\mathrm{v}_m)\in \mathbb{R}^m$ and for a.e. $t\in[0,\wT]$. 

\begin{remark}
Notice that the matrix $\mathbb{L}_{\lambda(t)}$ is symmetric by \eqref{eq F0i+Fij} and Jacobi identity.  
\end{remark}

We strengthen the two necessary conditions above defined. 
\begin{ass} {\bf(High Order Goh Condition)} \label{ass GGoh}
We assume that the reference extremal $\widehat{\lambda}$ satisfies the following equations
\[
\langle \wla(t), f (\wxi(t))\rangle =0 \qquad \forall \, f\in \Lie{\mathfrak{f}},\  t\in[0,\wT] .
\]
\end{ass}

HOGC is a stronger condition than the usual one, but in our case the optimality of the singular extremal is proved also when $U=\mathbb{R}^m$; in Appendix
\ref{app variations} we show that, for $U=\mathbb{R}^m$, HOGC is a necessary optimality condition. 
As a matter of fact,
if the Lie algebra generated by the controlled vector fields is 2-step bracket generating, then
HOGC coincides with Goh condition.

Remark moreover that, under Assumption~\ref{ass GGoh}, the quadratic form $\mathbb{L}_{\wla(t)}$ is given by
\begin{equation} \label{eq: L}
\mathbb{L}_{\wla(t)}: \mathbf{v} \mapsto \sum_{i,j=1}^m v_i v_j F_{ij0} ({\wla}(t)),
\end{equation}
so that it is continuous as a function of time.

\begin{ass} {\bf (Strengthened Generalized Legendre Condition)} \label{ass SGLC}
There exists a constant $c>0$ such that 
\begin{equation} \label{eq: SGLC}
\mathbb{L}_{\wla(t)} [\mathbf{v}]^2 \leq - c |\mathbf{v}|^2
\end{equation}
for any $\mathbf{v}=(\mathrm{v}_1,\ldots,\mathrm{v}_m)\in \mathbb{R}^m$ and for every $t\in[0,\wT]$.
\end{ass}

As a consequence of Assumptions \ref{ass lie}--\ref{ass GGoh} and equation \eqref{eq F0i+Fij}, we get that 
\begin{align}
  \langle \wla(t),[f_0,f](\wxi(t))\rangle&=0 \qquad  \forall \, f 
\in \Lie{\mathfrak{f}}, \quad \forall \, t\in[0,\wT] \label{eq: reg3}\\
\sum_{j=1}^m (\mathbb{L}_{\wla(t)})_{ij} \widehat{u}_j(t)&=F_{00i} (\wla(t)) \qquad i=1,\ldots,m \quad \mathrm{a.e.} \ t\in[0,\wT].
\label{F00i+Lu}
\end{align}
From \eqref{F00i+Lu} and Assumption \ref{ass SGLC} we can recover the reference control as feedback on the cotangent bundle and, by induction, we can 
prove that it is smooth.

From now on we restrict to a (full-measure) neighborhood $\U$ of $\wla([0,\wT])$ in $T^*M$ where SGLC is satisfied, that is, where the quadratic form $\mathbb{L}_{\ell}$ 
is negative-definite. 
We define two submanifolds of $\U$ which are crucial for our construction: 
\begin{align}
\Sigma&=\{\ell \in \U   : \langle \ell,f(\pi \ell)\rangle=0\ \forall \,  f 
\in \Lie{\mathfrak{f}} \} \label{eq: sigma}  \\
\esse&=\{\ell \in \Sigma  :  \langle \ell,[f_0,f] (\pi \ell) \rangle=0 \ \forall \,  f 
\in \Lie{\mathfrak{f}}\} \label{eq: S}. 
\end{align}

By Assumption \ref{ass lie}, possibly restricting $\U$, $\Sigma$ is an embedded manifold of codimension $R$.
Moreover every singular extremal that satisfies HOGC is contained in $\esse$. We set the following regularity assumption on $\esse$, 
which requires that it is a submanifold of maximal dimension (see the
arguments below).

\begin{ass}[\bf Regularity of $\esse$] \label{ass F_0i=0}
The manifold $\esse$ has constant codimension $m$ in $\Sigma$.
\end{ass}

Thanks to regularity assumptions, the manifolds $\Sigma$ and $\esse$ have the following properties. The proofs can be obtained adapting 
those in \cite{ChiSte}.

\begin{enumerate}[(P1)]
\item It is easy to see that the Lie algebra Lie$_{\ell}(\vec{F}_1,\ldots,\vec{F}_m)$ 
has dimension $R$ for every $\ell \in \Sigma$. Moreover, every vector field $X \in $ Lie$_{\ell}(\vec{F}_1,\ldots,\vec{F}_m)$
is tangent to 
$\Sigma$.
\item It is not difficult to prove that SGLC implies 
that the vector fields $f_{01},\ldots,f_{0m}$ are linearly independent, and their span is transversal to
Lie$(\mathfrak{f})$. Therefore
$\esse$ has codimension at least $m$ in $\Sigma$, and Assumption \ref{ass F_0i=0} states then that $\esse$ has the maximal dimension. The arguments above prove also
that
$\esse$ it can be characterized by
\[
\esse=\{\ell\in \Sigma: F_{0i}(\ell)=0  \ \forall \, i=1,\ldots,m\}.
\]
Notice that the existence of a normal singular extremal satisfying HOGC implies that $R+m\leq n-1$ and that $f_0$ is transversal to Lie$(\mathfrak{f})$,
in a neighborhood of the corresponding trajectory on $M$.
Moreover, Assumption \ref{ass F_0i=0} is equivalent to the following one:
\[[f_0,f]\in \Lie{\mathfrak{f}} + \mathrm{span}(\{f_{01},\ldots,f_{0m}\}) \qquad \forall \, f \in \Lie{\mathfrak{f}}.\]
\item Similar arguments show that $\vec{F}_{01},\ldots,\vec{F}_{0m}$ are transversal to $\Sigma$, and $\vec{F}_1,\ldots,\vec{F}_m$ are transversal to $\esse$. 
\item $\vec{F}_0$ is tangent to $\Sigma$ in $\esse$.
\item 
Our assumptions guarantee the existence of  a Hamiltonian vector field tangent to all singular extremals.
Indeed, setting for every $\ell \in \U$
\begin{equation} \label{feedback}
\boldsymbol{\nu}(\ell)=\mathbb{L}^{-1}_{\ell} 
\begin{pmatrix}
F_{001}(\ell)\\
\vdots\\
F_{00m}(\ell)
\end{pmatrix},
\end{equation}
we get that the vector field $\vec{F}_{\esse}=\vec{F}_{0}+\sum_{i=1}^m \nu_i \vec{F}_i$ is tangent to $\esse$, 
and the reference extremal $\wla(\cdot)$ is an integral curve of $\vec{F}_S$. Indeed every singular extremal associated with our
dynamics is an integral curve of $\vec{F}_S$.
\end{enumerate}

\section{Second variation} \label{sec secvar}
In this section we define the second variation for the problem under study, and we investigate the consequences of the coercivity of the second variation.
The computations can be recovered by adapting those present in \cite{PoggStebsb,SteRoma}. 

\subsection{Construction of the second variation}
We consider the second variation associated with the sub-problem of \eqref{eq: min T}-\eqref{eq: contr sys} with fixed 
final point, that is we add the constraint $\xi(T)=\wq_f$.
To compute this second variation, we transform such minimum-time problem into a
Mayer problem on the fixed time interval $[0, \wT]$ and on the
state space $\mathbb{R} \times M$.  Namely, putting $u_0$ as a new constant control
with positive values, we reparametrize the time as $u_0 t$, and we set $\boldsymbol{q} = (q^0, q)\in\mathbb{R}
\times M$,  ${\boldsymbol f}_0 (\boldsymbol{q}) = f_0(q) + \frac{\partial}{\partial q^0}$ and ${\boldsymbol f}_i ({\boldsymbol q})
 = f_i(q)$, $i=1,\ldots,m$. 
Then the minimum-time problem between $N_0$ and $\wq_f$ is equivalent to the Mayer
problem on $\mathbb{R} \times M$ described below.
\begin{equation}\label{eq: sistc}
  \min  \xi^0(\wT)
\end{equation}
subject to
\begin{equation}
\begin{cases}\label{eq: sistc1}
\dot{\boldsymbol{\xi}}(t) = u_0  \boldsymbol{f}_0(\boldsymbol{\xi}(t)) + \sum_{i=1}^m u_0
  u_i(t)  \boldsymbol{f}_i(\boldsymbol{\xi}(t)) \quad t \in [0, \wT]    \\
 \boldsymbol{\xi}(0) \in \{0\} \times N_0, \quad
  \boldsymbol{\xi}(\wT) \in \mathbb{R} \times \{ \wq_f\}     \\
   (u_0, \mathbf{u}) \in (0, +\infty) \times L^{\infty}([0,\wT],U)
\end{cases}
\end{equation}
where $\boldsymbol{\xi}=(\xi^0,\xi) $.
It is not difficult to see that the trajectory $\widehat{\boldsymbol{\xi}}(t) = (t,\wxi(t)) $, 
associated with the controls $u_0= 1$ and $\bf{u}=\widehat{u}$, 
is an extremal with associated adjoint covector 
$\widehat{\boldsymbol{\lambda}} :
s\mapsto ( (-1, t), \wla(t)) \in \mathbb{R}^* \times T^*M$.

For $t \in [0,\widehat{T}]$, we define the evolution map
$\widehat{S}_t : M  \rightarrow M $ by its action $\widehat{S}_t : x_0  \mapsto \xi(t)$, where
$\xi$ is the solution of the equation $\dot{\xi} = f_0(\xi) + \sum_{i=1}^m \widehat{u}_i f_i(\xi)$ with initial condition $\xi(0)=q_0$. In particular, $\widehat{S}_t(\wq_0)= \widehat{\xi}(t)$.
We locally define around $\wq_0$ the \emph{pull-back vector fields} 
\[
g_t^i = \widehat{S}_{t*}^{-1} f_i \circ \widehat{S}_t,\qquad i=0,\ldots,m. 
\]
Analogously, for the Mayer problem we define the evolution $\widehat{\bf S}_t : \mathbb{R} \times M \to\mathbb{R} \times M$ 
as $\widehat{\bf S}_t : (q^0, q) \mapsto (q^0+t,\widehat{S}_t (q))$,
and
the 
pull-back system of \eqref{eq: sistc1} corresponding to the reference control $\widehat{\mathbf{u}}$ as 
\[\boldsymbol{\eta}(t)=\widehat{\bf S}_t^{-1} \circ \boldsymbol{\xi}(t).\]
The Mayer problem \eqref{eq: sistc}-\eqref{eq: sistc1} is then equivalent to the following one:
\[
\min \eta^0(\wT)
\]
subject to the control system
\begin{equation}
\begin{cases}
\dot{\eta}^0(t)=u_0-1\\
\dot{\eta}(t)=(u_0 -1) g_t^0(\eta(t)) + \sum_{i=1}^m(u_0 u_i(t)-\widehat{u}_i(t)) g_t^i(\eta(t))\\
\boldsymbol{\eta}(0)\in \{0\} \times N_0 \qquad \boldsymbol{\eta}(\wT)\in \mathbb{R} \times \{\wq_0\}.
\end{cases}
\end{equation}

Let us now consider variations  $(\delta u_0,\delta x,\delta u)\in \mathbb{R} \times T_{\wq_0}N_0 \times L^{\infty} ([0,\widehat{T}],\mathbb{R}^m)$ 
around the  
reference trajectory, and
let us evaluate the coordinate-free second variation of the Mayer problem,
following \cite{ASZ98}.
We choose any two smooth functions $\boldsymbol{\alpha},\boldsymbol{\beta} : \mathbb{R} \times M \to \mathbb{R}$ that satisfy the following constraints:
\begin{gather}
\boldsymbol{\alpha}(q^0,q)=\alpha(q)-q^0, \qquad \alpha|_{N_0} \equiv 0,  \qquad d\alpha(\wq_0) = \widehat{\lambda}(0) \label{eq: alpha},\\
\boldsymbol{\beta}(q^0,q)=q^0+\beta(q), \qquad  d\beta(\wq_0) =  - \widehat{\lambda}(0) , \label{eq: beta} 
\end{gather}
for two suitable smooth functions $\alpha,\beta : M \to \mathbb{R}$. 
Thanks to High Order Goh Conditions, we can choose the function $\alpha$ in such a way that it satisfies the constraint
$\alpha|_{\mathcal{I}_{\wq_0}}\equiv 0$, where $\mathcal{I}_{\wq_0}$ is the integral manifold of the distribution $\Lie{\mathfrak{f}}$ passing through $\wq_0$.
Moreover, we can choose $\beta=-\alpha$, since the second variation does not depend on the particular choice of $\alpha$ and $\beta$ with the properties \eqref{eq: alpha}
and \eqref{eq: beta} (see \cite{ASZ98}).

The second variation is given by
\begin{align}  
J^{\prime \prime} [(\delta x,\delta u_0,\delta u)]^2 &=  
\int_0^{\widehat{T}} 
\delta u_0 L_{\delta \boldsymbol{\eta}(t)} L_{\boldsymbol{g}_t^0} \widehat{\boldsymbol{\beta}}(\wT,\wq_0)+\sum_{i=1}^m (\delta u_0 \widehat{u}_i(t) + 
\delta u_i(t)) L_{\delta \boldsymbol{\eta}(t)} L_{\boldsymbol{g}_t^i} \widehat{\boldsymbol{\beta}}(\wT,\wq_0)  \;dt,
\label{eq: secvar sing}
\end{align}
\noindent 
where 
$\delta \boldsymbol{\eta}(t) \in \mathbb{R}\times T_{\wq_0}M $ is the linearization of $\boldsymbol{\eta}(t)$ and satisfies the following system:
\begin{equation} \label{eq: delta eta}
\begin{cases}
\dot{\delta \eta^0}(t)=\delta u_0\\
\dot{\delta\eta}(t) = \delta u_0 g_t^0(\wq_0) +\sum_{i=1}^m (\delta u_0 \widehat{u}_i(t)+ \delta u_i(t)) g_t^i(\wq_0) \\ 
\delta \boldsymbol{\eta} (0) = (0,\delta x) \in \{0\} \times T_{\wq_0} N_0, \quad \delta \boldsymbol{\eta}  (\widehat{T}) \in \mathbb{R} \times \{0\}.
\end{cases}
\end{equation}

\begin{remark}
If $\delta \boldsymbol{\eta}$ satisfies the system \eqref{eq: delta eta}, then the value of the second variation does not depend on the particular choice of $\alpha$
and $\beta$, provided that they satisfy properties \eqref{eq: alpha}-\eqref{eq: beta} (see \cite{ASZ98}). 
Then $J^{\prime \prime}$ is well defined and coordinate free.
\end{remark}

Since we are interested only in the so-called time-state local optimality, we restrict us to the
subproblem with $\delta u_0 = 0$, and, proceeding as in \cite{PoggStebsb}, we define $\boldsymbol{w}(\cdot)$ and $\boldsymbol{\epsilon}$ by
\begin{gather}
w_i(t)=   \int_t^{\wT} \delta u_i(s) \; ds \\
\epsilon_i = w_i(0),   \label{eq:numero2}
\end{gather}
for $i=1,\ldots,m$.
In this way, the control variation $\delta u$ is embedded as the pair 
$(\boldsymbol{\epsilon},\boldsymbol{w}(\cdot))$ in the space $\mathbb{R}^{m} \times L^2 ([0,\widehat{T}],\mathbb{R}^m)$.
We remark that this embedding is continuous and it has dense image.
Then the second variation defined by \eqref{eq: secvar sing}-\eqref{eq: delta eta} writes as 
\begin{align}  
J^{\prime \prime} [(\delta x,\boldsymbol{\epsilon},\boldsymbol{w}(\cdot))]^2 &= 
\frac{1}{2} \sum_{i,j=1}^m \left( L_{ \epsilon_if_i} L_{  \epsilon_j f_j} \beta(\wq_0) + 
\int_0^{\widehat{T}} w_i(t) w_j(t) L_{[\dot{g}_t^i,g_t^j]}   \beta(\wq_0) \;dt \right) \nonumber \\
&+ \sum_{i=1}^m \left( L_{\delta x}L_{ \epsilon_if_i} \beta(\wq_0) +     
\int_0^{\widehat{T}} w_i(t) L_{\zeta(t)} L_{\dot{g}_t^i} \beta(\wq_0) \; dt \right).\label{eq: secvar}
\end{align}
where
the function $\zeta :[0,\wT] \to T_{\wq_0}M $ is the solution of the equation
\begin{equation} \label{eq: zeta dot}
\dot{\zeta}(t) = \sum_{i=1}^m  w_i(t) \dot{g}_t^i(\wq_0),
\end{equation}
\noindent
with boundary conditions
\begin{equation} \label{eq: zeta bc1}
\zeta(0) = \delta x + \sum_{i=1}^m \epsilon_i f_i(\wq_0) , \qquad \zeta({\widehat{T}}) =0.
\end{equation}

Let us observe that the second variation is realized as a linear-quadratic control problem in the state-variable $\zeta$, with control $\boldsymbol{w}$ 
(see \cite{ASZ98,SteRoma,SteZez97}). 
Notice moreover that 
\[\dot{g}_t^i(\wq_0)=\widehat{S}_{t*}^{-1}[f_0,f_i]\circ \wxi(t)
\quad \mbox{ and } \quad
L_{[\dot{g}_t^i,g_t^j]}   \beta(\wq_0) = -F_{ij0}(\wla(t)).\] 
Finally, thanks to the choice of $\beta$, the finite-dimensional term in \eqref{eq: secvar} is null.

It is clear that, if $T_{\wq_0}N_0\cap \mathrm{span}(\{f_1(\wq_0),\ldots,f_m(\wq_0)\})\neq 0$, then 
the above defined quadratic form cannot be coercive.
On the other hand, the paradigm exposed in the introduction requires that the flow of 
the super-Hamiltonian emanating from $\Sigma$
remains contained in $\Sigma$, as Theorem \ref{cond suff lib-fix}
in Section \ref{sec results} describes (see also \cite{PoggStebsb,SteRoma});
in particular, the horizontal Lagrangian sub-manifold of the initial constraints must be contained in $\Sigma$.
These arguments suggest to require coercivity of \eqref{eq: secvar} allowing all vectors in $\mathrm{Lie}_{\wq_0}(\mathfrak{f})$ 
to be initial conditions for the state $\zeta$; see Section \ref{sub 2var} and the proof of Theorem \ref{th: result}
in Section \ref{sec results}. In particular, this means to consider the second variation for the minimum time problem from
$ \mathcal{I}_{\wq_0}$ to $\wq_f$.

This is not surprising, since our result holds true also for unbounded controls; indeed, if $U=\mathbb{R}^m$, for every two points $q_1,q_2 \in \mathcal{I}_{\wq_0}$, there exist a sequence of times $t_k\to 0$ and a sequence of controls ${\bf u}_k$
such that the sequence $t_k\mapsto \{\xi_k(q_1,{\bf u}_k,t_k)\}_k$ of the solutions at time $t_k$ of the system \eqref{eq: contr sys}
starting from the point $q_1$ and relative to the control ${\bf u}_k$ tends to $q_2$ (see \cite[Lemma 4.1, Corollary 4.1, Remark 4.1]{BacSte}). 
As a consequence, if the controls are not bounded,
the infimum of the time for moving inside $\mathcal{I}_{\wq_0}$ is zero, therefore we can think
that, in some sense, if $T$ is the minimum time for joining $\wq_0$ with $\wq_f$ , then it is also
the infimum of the time for reaching $\wq_f$ from $\mathcal{I}_{\wq_0}$.

This digression suggests us the suitable coercivity assumption for this problem.
\begin{ass} \label{ass coercivity}
For $(\boldsymbol{\epsilon},\boldsymbol{w}(\cdot)) \in \mathbb{R}^R \times L^2([0,\wT],\mathbb{R}^m)$, let $\zeta(\cdot)$
be the solution of the control system
\begin{equation}
\begin{cases} \label{eq: zeta extended}
\dot{\zeta}(t) = \sum_{i=1}^m  w_i(t) \dot{g}_t^i(\wq_0)\\
\zeta(0) = \sum_{i=1}^R \epsilon_i f_i(\wq_0),
\end{cases}
\end{equation}
where $f_{m+1},\ldots,f_R$ are  some locally defined vector fields chosen to complete the basis for 
Lie$(\mathfrak{f})$ in a neighborhood of 
$\wq_0$. 
The quadratic form 
\begin{equation}  
J^{\prime \prime} [(\boldsymbol{\epsilon},\boldsymbol{w}(\cdot))]^2 = 
\frac{1}{2} \sum_{i=1}^m \left(\int_0^{\widehat{T}} 2w_i(t) L_{\zeta(t)} L_{\dot{g}_t^i} \beta(\wq_0) +\sum_{j=1}^m w_i(t) w_j(t) L_{[\dot{g}_t^i,g_t^j]}   
\beta(\wq_0) \;dt \right) \label{eq: secvar fin}
\end{equation}
is coercive on the subspace $\mathcal{W}$ of $\mathbb{R}^R \times L^2([0,\wT],\mathbb{R}^m)$ defined by the constraint $\zeta(\wT)=0$.
\end{ass}
\noindent
\subsection{Consequences of the coercivity assumption}\label{sub 2var}

Let us now introduce a special coordinate frame in a neighborhood of $\wq_0$, 
completing the set $\{f_1,\ldots,f_m\}$ 
with $n-m$ locally defined vector fields $f_{m+1},\ldots,f_n$ such that $\{f_1,\ldots,f_n\}$ 
is a basis for $T_qM$, and $\{f_1,\ldots,f_R\}$ is a basis for Lie$(\mathfrak{f})$,
in a neighborhood of
$\wq_0$.
The coordinate frame is the inverse of the map $\Upsilon : \mathbb{R}^n \to M$ defined as
\begin{equation} \label{frame}
\Upsilon(x_1,\ldots,x_n) = \exp(x_1f_1)\circ\exp(x_2f_2)\circ \cdots \circ \exp(x_nf_n) (\wq_0).
\end{equation}

In particular, $\Upsilon^{-1}(\wq_0)=(0,\ldots,0)$,
for $j=1,\ldots,n$ we have
\[
\frac{\partial}{\partial x_j} \Big|_{(0,\ldots,0)}=f_j(\wq_0)
\]
 and
\begin{equation} \label{eq:framesigma}
L_fx_i\equiv 0 \, ,\ \forall f\in \mathrm{Lie}(\mathfrak{f})\, ,\ i=R+1,\ldots,n.
\end{equation}

If we denote with
$(\widehat{p}_1,\ldots,\widehat{p}_n)$ the coefficients of $\well_0$ in this coordinate frame, then
it is easy to see that $\well_0=\sum_{i=R+1}^n \widehat{p}_i dx_i$. 

Define the symmetric 2-form 
\[
\Omega=\frac{1}{2}\sum_{i=R+1}^n dx_i \otimes dx_i,
\]
and extend it on the whole $T_{\wq_0}M \times L^2([0,\wT],\mathbb{R}^m)$
putting $\Omega [(\delta x,\boldsymbol{w}(\cdot))]^2=
\frac{1}{2}\sum_{i=R+1}^n (\delta x_i)^2$.

Then, if Assumption \ref{ass coercivity} holds, we can apply \cite[Theorem 13.2]{hestenes} to conclude that there exists a $\rho>0$ such that the form
\begin{equation} \label{J rho}
\J_{\rho}= J^{\prime \prime}+\rho \Omega 
\end{equation}
is coercive on the subspace $\widetilde{\mathcal{W}}\in T_{\wq_0}M \times L^2([0,\wT],\mathbb{R}^m)$ of the 
variations such that the solutions of the system \eqref{eq: zeta dot} satisfy
\begin{equation} \label{eq: zeta bc2}
\zeta(0) \in T_{\wq_0} M, \qquad \zeta(\wT)=0.
\end{equation}
Namely, $\J_{\rho}$ represents the second variation of the linear-quadratic problem associated with the variable $\zeta$, with 
\emph{free} initial condition and fixed final condition.

The Hamiltonian $H_t^{\prime \prime}: T^*_{\wq_0}M\times T_{\wq_0}M \to \mathbb{R}$ associated with this linear-quadratic problem is given by 
\begin{equation} \label{H''}
H_t^{\prime \prime}(\omega,\delta x)
=
\frac{1}{2}\mathbb{L}_{\wla(t)}^{-1} 
\left[\begin{pmatrix}
\langle \omega, \dot{g}_t^1 \rangle  +  L_{\delta x} L_{\dot{g}_t^1} \beta (\widehat{q}_0) \\
\vdots \\
\langle \omega, \dot{g}_t^m \rangle  + L_{\delta x} L_{\dot{g}_t^m} \beta (\widehat{q}_0)
\end{pmatrix}\right]^2,
\end{equation}
where  and $\mathbb{L}_{\wla(t)}^{-1}$ has to be thought of as a
quadratic form on $\mathbb{R}^m$, for every $t$ (see \cite{ASZ98,ChiSte}).

Set
\begin{equation} \label{L primo primo}
L^{\prime \prime}=\{(-2\rho\Omega (\delta x,\cdot),\delta x) : \delta x \in T_{\wq_0}M\}.
\end{equation}
The quadratic form \eqref{J rho} is coercive on the space $\widetilde{\mathcal{W}}$ if and only if
\begin{equation} \label{new inclusion}
\ker \pi \mathcal{H}_t^{\prime \prime} |_{L^{\prime \prime}} =\{0\}
\end{equation}
for every $t\in [0,\wT]$ (see \cite{SteZez97}).

\section{Geometry near the reference extremal} \label{sec geometry}

In this section we state the geometric properties of the vector fields and the Hamiltonians linked to
our system and we define the super-Hamiltonian, stating its properties.

In a neighborhood of the reference extremal, $\Sigma$ can be described as $\esse\times [-\epsilon,\epsilon]^m$, for some $\epsilon>0$. 
Indeed, for $\bt=(t_1,\ldots,t_m) \in \mathbb{R}^m$, let us denote with $\bt\bff$ the vector field
$\sum_{i=1}^m t_i\vec{F}_i$,  and let us consider the map
$\ell\mapsto \exp(\bt\bff)(\ell)$ (that is, the solution at time $t=1$ of the equation $\dot{\ell}=\sum_{i=1}^m t_i \vec{F}_i (\ell)$). For a sufficiently small $\epsilon>0$,
 the map 
$$\tpsi: (\ell,\bt)\in \U \times [-\epsilon,\epsilon]^m \mapsto 
\exp(\bt\bff)(\ell)\in T^*M$$
is well defined and it is easy to prove the following:

\begin{proposition} \label{prop: psi}
Possibly restricting $\U$, there exists an $\epsilon>0$ such that  
$\psi: \esse  \times [-\epsilon,\epsilon]^m\to \Sigma $ 
is a diffeomorphism.
\end{proposition}

\begin{proof}
Since the whole Lie algebra generated by the $\vec{F}_i$ is tangent to $\Sigma$, then the range of the map $\psi$ restricted to $\esse\times [-\epsilon,\epsilon]$ 
is 
contained in $\Sigma$.
The thesis follows by compactness of the interval $[0,\wT]$, since
$D\psi (\ell,\textbf{0})=\mathrm{id}\times(\vec{F}_1(\ell),\ldots,\vec{F}_m(\ell))$ has maximal rank (see property (P3) at the end of Section
\ref{not e ass}).
\end{proof}

\begin{remark} \label{rem tpsi}
It is easy to see that
\begin{equation}\label{jac psi}
\partial_{t_i} \tpsi(\ell,\boldsymbol{0})=\vec{F}_i (\ell),\qquad i=1,\ldots,m. 
\end{equation}
and that for each function $F$ defined on $\U$
\begin{equation}\label{second psi}
\partial^2_{t_it_j}(F\circ \tpsi)(\ell,\boldsymbol{0})=\frac{1}{2}
\left(L_{\vec{F}_i} L_{\vec{F}_j}+L_{\vec{F}_j} L_{\vec{F}_i} \right)F(\ell),
\qquad i,j=1,\ldots,m. 
\end{equation}
We remark also that $\tpsi$ maps $\Sigma \times [-\epsilon,\epsilon]^m$ into $\Sigma$.
\end{remark}

For singular extremals, the maximized Hamiltonian is well-defined and coincides with $F_0$, but
its associated Hamiltonian vector field is multi-valued: indeed, all the
Hamiltonians of the form $F_0 +\sum_{i=1}^m u_i F_i$, $\mathbf{u} \in U$, coincide and realize the maximum along any extremal contained in $\Sigma$. 
Moreover, no selection of such multi-valued Hamiltonian vector fields is suitable 
to construct the field of non-intersecting state-extremals that we will use to compare the costs associated with the candidate trajectories. 
For an insight in the single-input case, see \cite[Section 4]{PoggStebsb} .
Then, as already done in \cite{ChiSte,PoggStebsb,SteCDC04},
we substitute the maximized Hamiltonian $F_{max}$ with a 
the time-dependent
super-Hamiltonian $H_t=H_0+\sum_{i=1}^m \widehat{u}_i(t) F_i$, where $H_0$ is defined as described below.

The first step is to define a suitable map which turns out to project $\Sigma$ onto $\esse$.
\begin{lemma} \label{lemma: theta}
Possibly restricting $\U$, there exist $m$ smooth functions $\vartheta_i : \U \to \mathbb{R},\ i=1,\ldots,m$, such that, 
denoting with 
$\bth\bff(\ell)$ the vector field $\sum_{i=1}^m\vartheta_i\vec{F}_i$, the map  
\begin{equation*} 
\phi: \ell \in \U \mapsto \exp(\bth\bff)(\ell) \in \U
\end{equation*}
satisfies 
\begin{equation}\label{eq: theta1}
\phi(\Sigma)\subset \esse. 
\end{equation}
Moreover, for every $\ell \in \esse$,  $\delta \ell \in T_{\ell}T^*M$, it holds
\begin{equation} \label{dtheta}
\begin{pmatrix}
\langle d \vartheta_1(\ell),\delta \ell\rangle\\
\vdots\\
\langle d \vartheta_m(\ell),\delta \ell\rangle
\end{pmatrix}=
\mathbb{L}_{\ell}^{-1} 
\begin{pmatrix}
\langle d F_{01}(\ell),\delta \ell\rangle\\
\vdots\\
\langle d F_{0m}(\ell),\delta \ell\rangle
\end{pmatrix}.
\end{equation}
\end{lemma}
\begin{proof}
 Let us consider the function $\Phi: \U \times \mathbb{R}^m\to \mathbb{R}^m$ defined by
\[
\Phi(\ell,\bt)=
\begin{pmatrix}
F_{01} \circ \tpsi(\ell,\bt)\\
\vdots\\
F_{0m} \circ \tpsi(\ell,\bt) 
\end{pmatrix}.
\]
Notice that for every $\ell \in \esse$ we have 
$\tpsi(\ell,\bzero)=\ell$ and then $\Phi(\ell,\boldsymbol{0})=\boldsymbol{0}$.
Moreover, using \eqref{jac psi}, it is easy to show that  $\partial_{\bt}\Phi(\ell,\bzero)=-\mathbb{L}_{\ell}$,
for all $\ell \in \esse$, and then it has rank $m$.  
The implicit function theorem and the compactness of the interval $[0,\wT]$ ensure the existence of $m$ 
smooth functions $\vartheta_1,\ldots,\vartheta_m$, defined in a neighborhood of the reference extremal,
such that 
\begin{equation}\label{eq:imp}
F_{0i}\circ\phi(\ell)\equiv 0\, ,\ i=1,\ldots,m\, .  
\end{equation}
 Without loss of generality, we can assume that this 
neighborhood is $\U$.
Since $\exp(\bth\bff)(\ell) \in \Sigma$ for every $\ell \in \Sigma$, then \eqref{eq:imp} implies \eqref{eq: theta1}.

Fix $\ell \in \esse$; recalling that $\bth(\ell)=\textbf{0}$, thanks to \eqref{eq:imp} and \eqref{jac psi}, 
we get for every $\delta \ell \in T_{\ell} T^*M$ and $i=1,\ldots,m$
\begin{align*}
0&=
\langle dF_{0i}(\phi(\ell)), \partial_{\ell}\tpsi(\ell,\bt)|_{\bt=\bth(\ell)}  \delta \ell\rangle
+ \sum_{j=1}^m \langle dF_{0i}(\phi(\ell)),\partial_{t_j} \tpsi(\ell,\bt)|_{\bt=\bth(\ell)} \rangle
 \langle d \vartheta_j(\ell),\delta \ell \rangle\\
&=\langle dF_{0i} (\ell) , \delta \ell\rangle + \sum_{j=1}^m F_{j0i} (\ell) \langle d \vartheta_j(\ell),\delta \ell \rangle
\end{align*}
and then \eqref{dtheta}.
\end{proof}
\begin{remark}
It is easy to see that, as a particular case of \eqref{dtheta}, we get
\begin{equation} \label{derivata theta}
\langle d\vartheta_i(\ell),\vec{F}_j(\ell)\rangle=-\delta_{ij} \qquad \forall \ \ell \in \esse, \quad i,j=1,\ldots,m.
\end{equation}
  \end{remark}
The Hamiltonian $H_0$ is defined by means of the map $\phi$, as follows. 
\begin{definition} \label{def H_0}
We define the Hamiltonian function $H_0: \U \to \mathbb{R}$ as
\[
H_0(\ell)= F_0 \circ \phi  (\ell)
\]  
and we set
\[\chi = H_0-F_0.\]
\end{definition}
In order to prove the properties of $H_0$, we start by computing the first and second derivatives of $\chi$ on $\esse$.
\begin{proposition} \label{prop chi}
For every $\ell\in\esse$ and $\delta\ell\in T_{\ell}T^*M$, it holds
 $\langle d\chi(\ell),\delta \ell  \rangle =0$ and
\begin{equation} \label{hessiano chi}
 d^2\chi (\ell)[\delta \ell]^2= -  \sum_{r,s=1}^m  (\mathbb{L}_{\ell}^{-1})_{rs} \langle dF_{0r}(\ell),\delta \ell  \rangle
\langle dF_{0s}(\ell)	,\delta \ell  \rangle.
\end{equation}
\end{proposition}
\begin{proof}
 By definition we get 
\begin{equation}\label{eq:dchi}
d\chi(\ell)= dF_{0}(\phi(\ell))\left( \partial_{\ell}\tpsi(\ell,\bt)|_{\bt=\bth(\ell)} 
+ \sum_{j=1}^m \partial_{t_j} \tpsi(\ell,\bt)|_{\bt=\bth(\ell)}
 d \vartheta_j(\ell)\right)- dF_0(\ell)\, .
\end{equation}
For $\ell\in\esse$, by \eqref{jac psi}, we obtain
$
d\chi(\ell)= dF_{0}(\ell)\left( id 
+ \sum_{i=1}^m \vec {F}_i(\ell) d \vartheta_i(\ell)\right)- dF_0(\ell) =0
$.
Therefore $D^2\chi(\ell)$ is a well defined quadratic form for every $\ell\in\esse$.
To prove \eqref{hessiano chi}, we perform the computations in any chart.
Fix $\ell\in\esse$ and $\delta\ell\in T_{\ell}T^*M$ and recall \eqref{jac psi}, \eqref{second psi} and \eqref{dtheta};
from \eqref{eq:dchi} it follows
\begin{align*}
 D^2\chi(\ell)[\delta\ell]^2 & =D^2F_0(\ell)[\delta\ell +\sum_{i=1}^m \vec {F}_i(\ell)
\langle d \vartheta_i(\ell),\delta \ell \rangle  ]^2 - D^2F_0(\ell)[\delta\ell]^2+\\
 &+ dF_{0}(\ell)\left( \partial^2_{\ell\,\ell} \tpsi(\ell,\boldsymbol{0})
+  \sum_{i=1}^m \vec {F}_i(\ell)D^2\vartheta_i(\ell)\right)[\delta\ell]^2 +\\
&+dF_{0}(\ell)\left(2 \sum_{i=1}^m \partial^2_{\ell\, t_i}\tpsi(\ell,\boldsymbol{0})\delta\ell 
\langle d \vartheta_i(\ell),\delta \ell \rangle
+ \sum_{i,j=1}^m \partial^2_{t_j\, t_i}\tpsi(\ell,\boldsymbol{0})
\langle d \vartheta_j(\ell),\delta \ell \rangle \langle d \vartheta_i(\ell),\delta \ell \rangle
\right)\\
&=2\sum_{i=1}^m\left( D^2F_0(\ell)(\delta\ell , \vec {F}_i(\ell))
 +\langle D\vec {F}_i(\ell),\delta\ell\rangle
\right)\langle d \vartheta_i(\ell),\delta \ell \rangle +\\
&+ \sum_{i,j=1}^m\left(D^2F_0(\ell)(\vec {F}_i(\ell),\vec {F}_j(\ell)) + 
\langle dF_{0}(\ell),\partial^2_{t_j\, t_i}\tpsi(\ell,\boldsymbol{0})\rangle\right)
\langle d \vartheta_j(\ell),\delta \ell \rangle \langle d \vartheta_i(\ell),\delta \ell \rangle\\
&=2\sum_{i=1}^m L_{\delta \ell}L_{\vec {F}_i}F_0(\ell)\langle d \vartheta_i(\ell),\delta \ell \rangle+
\sum_{i,j=1}^m \partial^2_{t_j\, t_i}(F_0\circ\tpsi)(\ell,\boldsymbol{0})
\langle d \vartheta_j(\ell),\delta \ell \rangle \langle d \vartheta_i(\ell),\delta \ell \rangle\\
&= -2\sum_{i=1}^m
\langle dF_{0i} (\ell) , \delta \ell\rangle \langle d \vartheta_i(\ell),\delta \ell \rangle
 + \sum_{i,j=1}^m F_{ji0} (\ell) \langle d \vartheta_j(\ell),\delta \ell \rangle \langle d \vartheta_i(\ell),\delta \ell \rangle\\
&= - \sum_{i,j=1}^m F_{ji0} (\ell) \langle d \vartheta_j(\ell),\delta \ell \rangle \langle d \vartheta_i(\ell),\delta \ell \rangle
= -  \sum_{r.s=1}^m  (\mathbb{L}_{\ell}^{-1})_{rs} \langle dF_{0r}(\ell),\delta \ell  \rangle
\langle dF_{0s}(\ell)	,\delta \ell  \rangle.
\end{align*}
\end{proof}
Thanks to its definition, $H_0$ satisfies the following properties. 
\begin{theorem} \label{th: properties of H0}
The Hamiltonian $H_0$ has the following properties.
\begin{enumerate}
\item $F_0 =  H_0$  and $\vec{F}_0=\vec{H}_0$ on $\esse$.
\item The vector field $\vec{H}_0$ is tangent to $\Sigma$.
\item $F_0\leq H_0 $ on $\Sigma$.
\end{enumerate}
\end{theorem}
\begin{proof}
Since $\phi|_{\esse}$ is the identity, then $H_0=F_0$ on $\esse$. Moreover
\begin{align}
d H_0 (\ell) &= d F_0 \circ\left( \partial_{\ell} \tpsi(\ell,\bt)|_{\bt=\bth(\ell)} + \sum_{i=1}^m \partial_{t_i} \tpsi(\ell,\bt)|_{\bt=\bth(\ell)}
d\vartheta_i(\ell)\right) \label{eq: dH0} \\
&= d F_0 \circ\left( \mathrm{id} + \sum_{i=1}^m \vec{F}_i(\ell)d\vartheta_i(\ell) \right) =dF_0(\ell), \nonumber
\end{align}
since $\bth(\ell)=\boldsymbol{ 0}$ and $F_{0i}(\ell)= 0$ on $\mathcal{S}$. This ends the proof of (1). 

To prove (2), fix $\ell\in\Sigma$ and observe that $\partial_{t_i}\tpsi(\ell,\bt)\in$ Lie$(\mathfrak{f})(\tpsi(\ell,\bt))$,
$i=1,\ldots,m$, so that
by Assumption \ref{ass F_0i=0} we have
\[
\langle dF_0(\phi(\ell)),\partial_{t_i} \tpsi(\ell,t)|_{t=\bth(\ell)}\rangle=0.
\]
Therefore \eqref{eq: dH0} leads to
\[
dH_0(\ell)=dF_0(\phi(\ell))  \circ \partial_{\ell} \tpsi(\ell,\bt)|_{\bt=\bth(\ell)}, \qquad \ell\in \Sigma,
\]
and then we get easily
\[
\vec{H}_0(\ell) = \left[\partial_{\ell} \tpsi(\ell,\bt)|_{\bt=\bth(\ell)}\right]^{-1} \vec{F}_0(\phi(\ell)).
\]
Since $\phi(\ell)\in\esse$, $\vec{F}_0(\phi(\ell)) \in T_{\phi(\ell)}\Sigma$ and
$\left[\partial_{\ell} \tpsi(\ell,\bt)|_{\bt=\bth(\ell)}\right]^{-1}$ maps $T_{\phi(\ell)}\Sigma$ onto $T_{\ell}\Sigma$, 
then (2) is proved.

Statement (3) is a straight consequence of Proposition \ref{prop chi}, since for all
$\ell\in\esse$, $\delta\ell\in T_{\ell}\esse $ and $i,j=1,\ldots,m$, one gets
$D^2\chi(\ell)[\delta\ell]^2= D^2\chi(\ell)[\delta\ell, \vec{F}_i]=0$ and 
$D^2\chi(\ell)[\vec{F}_j, \vec{F}_i]=-F_{ij0}$.
 \end{proof}
We finally define the super-maximized Hamiltonian $H_t$ as follows. 

\begin{definition} 
We denote with $H_t$ the following time-dependent Hamiltonian:
\begin{equation} \label{eq: H_t}
H_t(\ell)= H_0(\ell) + \sum_{i=1}^m \widehat{u}_i(t) F_i(\ell),
\end{equation}
and with $\mathcal{H}_t$ the Hamiltonian flow generated by $H_t$.
\end{definition}
Notice that $\vec{H}_t$ is tangent to $\Sigma$ and to the reference extremal $\wla(\cdot)$.

\section{The result} \label{sec results}

In this section we state and prove our main result. 
It relies on the following Hamiltonian sufficient conditions, that we state and prove here below.

\begin{theorem}[\bf Geometric sufficient condition] \label{cond suff lib-fix}
Let $(\wxi,\wuu,\wT)$ be the an admissible triple for the minimum-time problem \eqref{eq: min T}-\eqref{eq: contr sys} with associate adjoint vector $\wla$, and let 
Assumptions \ref{ass lie}--\ref{ass F_0i=0}  be satisfied. 
Suppose that there exist a neighborhood $V$ of  
$\wq_0$ and a smooth function
 $\mathfrak{a} : V \to \mathbb{R}$ with the following properties:
\begin{enumerate} [(i)]
\item $d\mathfrak{a}(\wq_0)=\wla(0)$ 
\item the Lagrangian submanifold $\Lambda =\{d\mathfrak{a} (q): q\in V\}$ is contained in $\Sigma$ 
and satisfies
\begin{equation} \label{good proj}
\ker \pi_* \mathcal{H}_{t*} \cap T_{\widehat{\lambda}(0)} \Lambda =\{0\} \qquad \forall\;t\in[0,\wT].
\end{equation}
\end{enumerate}

Then $(\wxi,\wuu,\wT)$ is a strict strong-local minimizer for the minimum-time problem \eqref{eq: contr sys} between $\mathcal{I}_{\wq_0}$ and 
$\wq_f$ (or between $\wq_0$ and $\mathcal{I}_{\wq_f}$)
and a 
strong-local minimizer for the minimum-time problem \eqref{eq: contr sys} between $\mathcal{I}_{\wq_0}$ and $\mathcal{I}_{\wq_f}$.  
\end{theorem}

\begin{proof}
Consider the map 
\[
\mathrm{id} \times \pi\circ\mathcal{H} : (t,\ell) \in [0, \wT] \times \Lambda \mapsto (t, \pi\circ \mathcal{H}_t(\ell)) \in [0,\wT] \times M.
\]	
Hypothesis $(ii)$ and the compactness of the reference trajectory imply that (possibly restricting $V$) there exist 
a neighborhood $\mathcal{O}$ of the graph 
of $\wxi(\cdot)$ in $[0,\wT]\times \Lambda$ such that $\mathrm{id} \times \pi\circ\mathcal{H}$ is a diffeomorphism between 
$[0,\wT]\times \Lambda$ and $\mathcal{O}$.

Let $(\xi,\vv,T)$ be an admissible triple for the control system \eqref{eq: contr sys} such that its graph is contained 
in $\mathcal{O}$, 
$\xi(0)\in \mathcal{I}_{\wq_0}$, $\xi(T)\in \mathcal{I}_{\wq_f}$, and
$T\leq \wT$, and denote its lift on $\Lambda$ with $\ell(t)$, namely 
\[\ell(t)=(\pi\circ \mathcal{H}_t)^{-1} (\xi(t)).  
\]
Let us choose a curve $\mu_0 : [0,1]\to  \Lambda$ joining $\well_0$ with $\ell(0)$ satisfying  
$\pi \mu_0(t)\in \mathcal{I}_{\wq_0},\ \forall \ t\in[0,1]$,
and a curve $\mu_f : [T,\wT ] \to \Lambda$ joining $\well_0$ with $\ell(T)$ satisfying
$\pi \circ \mathcal{H}_t (\mu_f(t))\in \mathcal{I}_{\wq_f},\ \forall \ t\in[T,\wT]$.
Let us now define the following paths   in $[0,\wT] \times \Lambda$: 
\begin{align*}
\gamma_1 &= (t,\well_0)
&& t\in[0,\wT]\\
\gamma_2 &= (0, \mu_0(t)) && t\in[0,1]\\
\gamma_3 &= (t,\ell(t))   && t\in[0,T]\\
\gamma_4 &= \left(t,  \mu_f(t)\right) && t\in[T,\wT],
\end{align*}
and let $\gamma=(-\gamma_1) \cup \gamma_2 \cup \gamma_3 \cup \gamma_4$.

Consider the following 1-form on $[0,\wT] \times T^*M$
\begin{equation} \label{omega}
\omega(t,\ell)=\mathcal{H}_t^* \varsigma - H_t\circ \mathcal{H}_t (\ell) dt,
\end{equation}
where $\varsigma$ is the canonical Liouville form on $T^*M$. It is easy to prove (see \cite[Proposition 17.1]{AgSac}) that $\omega$ is exact on $[0,\wT] \times \Lambda$.
In particular, since $\gamma$ is a closed path contained in $[0,\wT] \times \Lambda$, then $\int_{\gamma} \omega=0$.
Let us now evaluate the single components of this integral.
It is easy to see that 
\[
\int_{\gamma_1} \omega=\int_{\gamma_2} \omega=0.
\]
Moreover
\begin{align*}
  \int_{\gamma_3} \omega  &= \int_0^T \langle \mathcal{H}_t (\ell(t)) , \dot{\xi}(t) \rangle - H_t \circ  \mathcal{H}_t (\ell(t))\; dt  \\
&=\sum_{i=1}^m \int_0^T   \langle \mathcal{H}_t(\ell(t)),(\mathrm{v}_i(t) - \widehat{u}_i(t)) f_i  (\xi(t)) \rangle \;dt-
\int_0^T \chi \circ  \mathcal{H}_t(\ell(t)) \;dt \leq 0.
\end{align*}
Therefore
 \[ 0\leq \int_{\gamma_4} \omega = \int_{T}^{\wT} \langle \mathcal{H}_t (\ell(t)), \pi_*  \circ \mathcal{H}_{t*} \dot{\mu}_f(t)\rangle \;dt  -
\int_{T}^{\wT} H_{t} \left( \mathcal{H}_t (\mu_f(t)) \right) \;dt.\]
The first term is zero, since $\pi_*  \circ \mathcal{H}_{t*} \dot{\mu}_f(t)$ is tangent to $\mathcal{I}_{\pi  \circ \mathcal{H}_{t} (\mu_f(t))}$, by
construction, and $ \mathcal{H}_t (\ell(t))$ is contained in $\Sigma$. 
Then 
\[
0\leq\int_{\wT}^{T} H_{t} \left(\mathcal{H}_t (\mu_f(t))  \right)\;dt =
 \int_{\wT}^{T} 1+\mathcal{O}(t)\;dt \leq  (T-\wT)+ o(T-\wT),
\] 
which is a contradiction, therefore $T\geq \wT$. 

If the final point $\wq_f$ is fixed, we claim that the reference extremal is strictly optimal. To prove the claim, we consider an admissible triple $(\xi,\vv,T)$
as above, such that $T=\wT$ and $\xi(T)=\wq_f$.

In particular, since in this case $\int_{\gamma_4}\omega=0$, we obtain that
\begin{align*}
0 &= \int_{\gamma_3} \omega =-\int_0^T \chi \circ  \mathcal{H}_t(\ell(t)) \;dt,
\end{align*}
which implies that $\lambda(t)=  \mathcal{H}_t(\ell(t)) \in \esse$ for every $t\in [0,T]$, then, in particular, that $\dot{\lambda}(t)\in T_{\lambda(t)}\esse$,
that is $\langle d F_{0j}(\lambda(t)),\dot{\lambda}(t)\rangle =0$ for $j=1,\ldots,m$.
By computations it is possible to show that 
\[
\dot{\lambda}(t) =\overrightarrow{\widehat{F}}_t(\lambda(t)) +\sum_{i=1}^m (\mathrm{v}_i(t)-\widehat{u}_i(t)) \, \mathcal{H}_{t*} (\pi \circ \mathcal{H}_{t})_*^{-1} f_i(\xi(t)).
\]

Thanks to \eqref{good proj}, possibly restricting $V$, we can find a family of smooth functions $\mathfrak{a}^t: V^t\to \mathbb{R}$, where $V^t$
is a neighborhood of $\wxi(t)$ and $\mathfrak{a}^0=\mathfrak{a}$,
such that $\mathcal{H}_t(\Lambda)=\{d\mathfrak{a}^t(q): q\in V^t\}$ for every $t\in[0,T]$. In particular, for every $t\in [0,T]$ we have that
\begin{gather*}
\lambda(t)=d\mathfrak{a}^t(\xi(t))\\
\mathcal{H}_{t*} (\pi \circ \mathcal{H}_{t})_*^{-1} f_i(\xi(t)) = d\mathfrak{a}^t_* f_i(\xi(t)), \quad i=1,\ldots,m.
\end{gather*}
Then 
\begin{align*}
\langle d F_{0j}(\lambda(t)),\mathcal{H}_{t*} (\pi \circ \mathcal{H}_{t})_*^{-1} f_i(\xi(t))\rangle = 
\sigma(d\mathfrak{a}^t_* f_i(\xi(t)) ,\vec{F}_{0j}(\lambda(t))) &= L_{f_{i}} L_{f_{0j}} \mathfrak{a}^t (\xi(t))\\ 
&= L_{f_{0j}} L_{f_{i}} \mathfrak{a}^t (\xi(t)) + L_{f_{i0j}} \mathfrak{a}^t (\xi(t)) \\
&=\langle \lambda(t),f_{i0j}(\xi(t))\rangle,
\end{align*}
since $ L_{f_{i}} \mathfrak{a}^t (\xi(t))$ is identically null, being $\mathcal{H}_t(\Lambda)=d\mathfrak{a}^t(V^t)$ contained in $\Sigma$.
This implies that, for every $j=1,\ldots,m$, it holds
\[
0= \langle d F_{0j}(\lambda(t)),\dot{\lambda}(t)\rangle = F_{00j}(\lambda(t)) + \sum_{i=1}^m \mathrm{v}_i(t) F_{i0j}(\lambda(t)),
\]
that is, in particular, that $\mathbf{v}$ is a solution of equation \eqref{feedback}, and therefore $\dot{\lambda}(t)=\vec{F}_{\esse}(\lambda(t))$. Since
 $\widehat{\lambda}(t)$ is solution of the same equation, and both $\lambda$ and $\wla$ pass through $\wq_f$, then they coincide.

The same argument shows that the reference triple is a strict strong-local minimizer for the minimum-time problem \eqref{eq: contr sys} between $\wq_0$ and $\mathcal{I}_{\wq_f}$.
\end{proof}

Now we state and prove the main result.

\begin{theorem} \label{th: result}
Let $(\wxi,\wu,\wT)$ be an admissible triple of the minimum-time problem \eqref{eq: min T}-\eqref{eq: contr sys} with associate adjoint 
vector $\wla$, and let $\wla$ be a normal singular extremal. If Assumptions \ref{ass lie}--\ref{ass coercivity} are satisfied, then $\wxi(\cdot)$ 
is a minimum-time trajectory between $\mathcal{I}_{\wq_0}$
and $\mathcal{I}_{\wq_f}$, and hence between $N_0$ and $N_f$.
Moreover, the reference trajectory is strictly optimal among all admissible trajectories between $\wq_0$ and $\mathcal{I}_{\wq_f}$ and 
among all admissible trajectories between $\mathcal{I}_{\wq_0}$ and $\wq_f$.
\end{theorem}

\begin{proof} 
The thesis comes straightaway once proved that the coercivity assumption (Assumption \ref{ass coercivity}) 
allows us to define a smooth function $\alpha_{\rho}$ that satisfies the hypotheses of Theorem \ref{cond suff lib-fix}. 
In particular, 
we define $\alpha_{\rho}$ in the adapted coordinates \eqref{frame} of Section \ref{sub 2var} as follows:
\[
\alpha_{\rho}(x)=\sum_{i=R+1}^n \widehat{p}_ix_i+\frac{\rho}{2}\sum_{i=R+1}^n x_i^2.
\]

It is easy to see that $\alpha_{\rho}$ satisfies property $(i)$ and that \eqref{eq:framesigma} implies
that $\Lambda$ is a Lagrangian 
submanifold contained in 
$\Sigma$. To prove \eqref{good proj}  we need to exploit
the links between the flow of the Hamiltonian $H_t^{\prime \prime}$ 
defined in equation \eqref{H''} and  $\mathcal{H}_{t*}$, as done in \cite{ChiSte,PoggStebsb,SteRoma}.

It is known that the pull-back flow $\mathcal{G}_t = \widehat{\mathcal{F}}_t^{-1}
\circ \mathcal{H}_t$ is the Hamiltonian flow relative to the Hamiltonian 
$G_t : T^*M \rightarrow  \mathbb{R}$ defined by 
\[
G_t =
\big(
H_t -  \widehat{F}_t \big) \circ \widehat{\mathcal{F}}_t=\chi \circ \widehat{\mathcal{F}}_t\]
(see \cite{marsden-ratiu}).
Since $D G_t(\well_0) = 0$, then $G_t^{\prime \prime}=\frac{1}{2}D^2G_t|_{\well_0}$ is a well defined quadratic form and its associated Hamiltonian flow is 
$\mathcal{G}_{t*} : T_{\well_0}
(T^*M) \rightarrow T_{\well_0}(T^*M)$.

Let $\beta= -\sum_{i=R+1}^n \widehat{p}_ix_i$; then
the linear map $\iota: T_{\wq_0}^*M \times T_{\wq_0}M \rightarrow T_{\well_0}(T^*M)$ as follows
\[
\iota(\omega,\delta x)=-\omega+d(- \beta)_*\delta x
\] 
establishes an anti-symplectic isomorphism between $T_{\wq_0}^*M \times T_{\wq_0}M$ and $T_{\well_0}(T^*M)$. In particular, 
it determines an equivalence between  the 
Hamiltonian functions $G_t^{\prime \prime}$ and $H^{\prime \prime}_t$, i.e.
the following identities hold:
\begin{align}
H_t^{\prime \prime} &= -G_t^{\prime \prime} \circ \iota  \label{equivalenza H G} \\ 
\overrightarrow{H_t^{\prime \prime}} &=\iota^{-1}\circ \overrightarrow{
G_t^{\prime \prime}} \circ \iota \nonumber \\
\mathcal{H}_t^{\prime \prime}&= \iota^{-1} \circ \mathcal{G}_{t*} \circ \iota.  \label{equivalenza2}
\end{align}
We need to prove only \eqref{equivalenza H G}, since the other two equations are a direct consequence
(see \cite{ChiSte} and references therein for details).

Consider $\ell \in \esse$, $\delta \ell \in T_{\ell}(T^*M)$, and set $\ell_t=\widehat{\mathcal{F}}_{t}(\ell)$. Then, thanks to \eqref{hessiano chi}, 
we have that 
\begin{align*}
D^2G_t(\ell)[\delta \ell]^2 &= D^2 \chi(\ell_t)\circ\widehat{\mathcal{F}}_{t*} \otimes \widehat{\mathcal{F}}_{t*}\\
&=-  \sum_{r.s=1}^m  (\mathbb{L}_{\ell_t}^{-1})_{rs} \langle dF_{0r}(\ell_t),\widehat{\mathcal{F}}_{t*}\delta \ell  \rangle
\langle dF_{0s}(\ell_t),\widehat{\mathcal{F}}_{t*}\delta \ell  \rangle\\
&=-(\langle dF_{01}(\ell_t),\widehat{\mathcal{F}}_{t*}  \delta \ell\rangle,\ldots,\langle dF_{0m}(\ell_t),
\widehat{\mathcal{F}}_{t*} \delta \ell\rangle) 
(\mathbb{L}^{-1}_{\ell_t})
\begin{pmatrix}
\langle dF_{01}(\ell_t),\widehat{\mathcal{F}}_{t*} 
\delta \ell\rangle\\ \vdots
\\ \langle dF_{0m}(\ell_t),\widehat{\mathcal{F}}_{t*} \delta \ell\rangle 
\end{pmatrix}\\
&= -2 H_t^{\prime \prime} \circ \iota^{-1}\delta \ell.
\end{align*}
By computations, it is easy to see that the space $L''$ defined in equation \eqref{L primo primo} satisfies the equality
$\iota L''=\{d \alpha_{\rho*}\delta x : \delta x\in T_{\wq_0}M\}=L$, therefore equations \eqref{new inclusion}
and \eqref{equivalenza2} imply
\[
\ker \pi_*  \mathcal{G}_{t*} |_{L}  = \{0\} \qquad \forall\ t\in[0,\wT].
\]
To end the proof it is sufficient to notice that $\widehat{\mathcal{F}}_{t*} $ is an isomorphism on
the vertical fibers, since it comes from a lifted Hamiltonian.
\end{proof}

\section{Examples} \label{sec example}

The classical Dubins and dodgem car problems concern the motion of a car on the plane $\mathbb{R}^2$ with constant speed and controlled (bounded) angular velocity
(see for instance \cite{AgSac,craven}). 
In particular, Dubins problem looks for minimum-time trajectories between fixed initial and final positions and orientations, while in the dodgem car problem the final 
orientation is free.

As shown in \cite{JurdjLie} (see also \cite{ChiStecdc}), this problem can be reformulated on the manifold $\mathbb{R}^2\times \mathrm{SO}(2)$, where SO(2) is the group 
of positively oriented rotations on $\mathbb{R}^2$. A great advantage of this formulation is that the extension to higher dimensions is straightforward: we denote 
with $(q,R)$ the elements of $\mathbb{R}^N\times$ SO($N$),  
where SO$(N)$ is the group of positively oriented rotations on $\mathbb{R}^N$, and we consider the control system 
\begin{equation} \label{extended dubins}
\begin{cases}
\dot{q}(t)= R(t)e_1\\
\dot{R}(t)= \sum_{j=1}^{N-1} u_j(t) R(t)A_j, \qquad |\mathbf{u}|\leq 1,  
\end{cases}
\end{equation}
where $e_1$ is the first element of the canonical basis of $\mathbb{R}^N$, and, for every $j=1,\ldots,N-1$, $A_j$ is the anti-symmetric matrix defined by
\[
(A_j)_{lm}=
\begin{cases}
-1 & \mathrm{if} \ l=1,m=j+1\\ 
1 & \mathrm{if} \ l=j+1,m=1\\
0 & \mathrm{otherwise}.  
\end{cases}
\] 
This system models the motion of a point in the $N$-dimensional space with constant speed equal to 1, 
where we control the orientation velocity. For both Dubins and dodgem car problems, 
the initial condition consists in fixing the initial point $q_0$ and the initial (unit length) velocity $\mathbf{v}_0$ of the trajectory on $\mathbb{R}^N$. Namely,
$N_0=\{q_0\} \times \mathcal{I}_{\mathbf{v}_0}$, where $\mathcal{I}_{\mathbf{v}}=\{S\in \mathrm{SO}(N): Se_1=\mathbf{v}\}$. 
Dubins problem looks for minimum-time trajectories joining $N_0$ with
$N_f=\{q_f\} \times \mathcal{I}_{\mathbf{v}_f}$, for some fixed $q_f\in \mathbb{R}^N$, $\mathbf{v}_f\in \mathbb{R}^N$ with $|\mathbf{v}_f|=1$, while dodgem car problem searches minimum-time trajectories from $N_0$ to 
$N_f=\{q_f\} \times$ SO$(N)$, for some fixed $q_f\in \mathbb{R}^N$.

The system \eqref{extended dubins} can be embedded in the matrix group GL$(N+1)$ (non singular $(N+1)$-dimensional matrices), via the map
\[
(q,R) \in \mathbb{R}^N \times \mathrm{SO}(N) \mapsto g=\begin{pmatrix}
                                            1 & 0\\ q & R 
                                            \end{pmatrix} \in \mathrm{GL}(N+1).
\]

This formulation is suitable also to consider the Dubins problem on other homogeneous spaces different from $\mathbb{R}^N$, that is the
$N$-dimensional sphere $S^N$ and the $N$-dimensional hyperbolic space $\mathbb{H}^N$,
defined as 
$\mathbb{H}^N=\{\boldsymbol{x}\in \mathbb{R}^{N+1}: -x_0^2+\sum_{i=1}^N x_i^2=1,\ x_0>0\}$.  
Briefly, a pair \emph{(point,orientation)} in $S^N\times$ SO$(N)$
can be represented in the group $G=$ SO$(N+1)$ in the following way: the first column of the matrix $g\in$ SO$(N+1)$ gives the coordinate representation
 (in $\mathbb{R}^{N+1}$)
of the point, the other $N$ columns determine an orthonormal frame in the tangent space to the sphere at the point. Analogously, 
a pair \emph{(point,orientation)} in $\mathbb{H}^N\times$ SO$(N)$
can be represented in the group $G=$ SO$(1,n)$, as above: the first column of the matrix $g$ gives the coordinate representation (in $\mathbb{R}^{N+1}$)
of the point, while the other $N$ columns determine an orthonormal frame in the tangent space.
More details on these representations can be found in Appendix \ref{Lie} and in \cite{Jurdjevic,JurdjLie}, where the authors study the geodetic problem for 
curves with bounded curvature.

We can then write Dubins problem on $M\in \{\mathbb{R}^n,S^N,\mathbb{H}^N\}$ in the following unified way: 
\begin{equation} \label{min T 2}
\min T 
\end{equation}
subject to
\begin{equation} \label{g dot eps}
\begin{cases}\dot{g}(t)=g(t)\begin{pmatrix}
           0 & -\varepsilon e_1^{\mathsf{T}} \\ e_1 & 0
         \end{pmatrix}
+ \sum_{i=1}^{N-1} u_i(t) g(t)\begin{pmatrix}
           0 & 0 \\ 0 & A_i
         \end{pmatrix}, \qquad |\mathbf{u}|\leq 1 \\ 
g\in G\\
g(0)\in N_0,\ g(T)\in N_f,
\end{cases} 
\end{equation} 
where $G$, $\varepsilon$ and the manifolds of the constraints depend on the manifold $M$ as shown below:

\begin{center}
\begin{tabular}{|c|c|c|c|c|}
\hline
$M$ & $\varepsilon$ & $G$ & $N_0$ & $N_f$\\
\hline
 $\mathbb{R}^N$ & $\varepsilon=0$ & $\mathbb{R}^N \rtimes$ SO$(N)$& $\left\{g  : g \mathfrak{e}_1=\begin{pmatrix}1 \\ \wq_0\end{pmatrix},\:
 g \mathfrak{e}_2=\begin{pmatrix}0 \\ \boldsymbol{v}_0\end{pmatrix}\right\}$
& $\left\{g : g \mathfrak{e}_1=\begin{pmatrix}1 \\ \wq_f\end{pmatrix} , \:
 g \mathfrak{e}_2=\begin{pmatrix}0 \\ \boldsymbol{v}_f\end{pmatrix}\right\}$\\
\hline
 $S^N$ & $\varepsilon=1$ & SO$(N+1)$& $\{g  : g \mathfrak{e}_1=\wq_0, \  g \mathfrak{e}_2=\boldsymbol{v}_0\}$
& $\{g  : g \mathfrak{e}_1=\wq_f, \ g \mathfrak{e}_2=\boldsymbol{v}_f\}$\\
\hline
 $\mathbb{H}^N$ & $\varepsilon=-1$ & SO$(1,N)$& $\{g  : g \mathfrak{e}_1=\wq_0, \: g \mathfrak{e}_2=\boldsymbol{v}_0\}$
& $\{g  : g \mathfrak{e}_1=\wq_f, \:  g \mathfrak{e}_2=\boldsymbol{v}_f\}$\\
\hline
\end{tabular}
\end{center}

\medskip

\noindent
where $\mathfrak{e}_i$ denotes the i-th element of the canonical basis of $\mathbb{R}^{N+1}$.

The control system \eqref{g dot eps} is a control-affine system of the form \eqref{eq: contr sys} with $m=N-1$ and $\dim M=N(N+1)/2$, 
and the corresponding left-invariant vector fields are defined as $f_i(g)=gA_i, \, i=0,\ldots,N-1$. We recall that
left invariant vector fields satisfy the following relation:
$g[A_i,A_j]=[f_i,f_j](g)$, where 
$[\cdot,\cdot]$ denotes also the usual matrix commutator.
Thanks to this equation, the commutation properties of the matrices $A_i$ extend also to their associated left-invariant vector fields.
In particular, the following properties are easily verified:

\begin{enumerate}[(i)]
 \item $\Lie{\{A_i : i=1,\ldots,N-1}$ is 2-step bracket-generating and isomorphic to $\mathfrak{so}(N)$ (Lie algebra of antisymmetric $N$-dimensional matrices); 
  $\mathfrak{so}(N)$ has dimension $R=N(N-1)/2$.
\item  
the matrices $\{[A_i,A_j]: i,j=1,\ldots,N-1\}$ generate the derived sub-algebra
$[\mathfrak{so}(N),\mathfrak{so}(N)]$,  which is
isomorphic to $\mathfrak{so}(N-1)$ and has dimension  $\frac{(N-1)(N-2)}{2}$
\item the matrices $\{A_0,[A_0,A_i],A_i,[A_i,A_j]: i<j=1,\ldots,N-1\}$ are linearly independent and form a basis for the Lie algebra of $G$.
\item for $i,j=1,\ldots,N-1$ the matrix commutators of the kind $[A_i,[A_j,A_0]]$ satisfy the following relations:
\begin{gather*}
[A_i,[A_i,A_0]]=-A_0\\
[A_i,[A_j,A_0]]=0 \quad \mathrm{if} \ i\neq j.
\end{gather*}
\item the matrices $\{A_0,A_{01},\ldots,A_{0m}\}$ mutually commute.
\item $A_0$ commutes with every element of $\{[A_i,A_j]: i,j=1,\ldots,N-1\}$.
\end{enumerate}
We notice that the submanifolds $N_0$ and $N_f$ are integral manifolds of the derived sub-algebra $\{[f_i,f_j]:i,j=1,\ldots,m\}$. 
Indeed, it is easy to verify that the Lie sub-algebra is contained
in the tangent spaces of $N_0$ and $N_f$; a dimensional computation proves the claim.

Let us now consider singular extremal for the problem \eqref{min T 2}-\eqref{g dot eps}. First of all, we
remark that this problem does not admit abnormal singular extremals, thanks to property (iii).
Moreover, thanks to property (v) and from equation \eqref{F00i+Lu}, we get that the reference control $\wu(\cdot)$ is identically zero, and that
the matrix $\mathbb{L}_{\wla(t)}=-F_0(\wla(t)) \mathbb{I}_n$.
In particular,  singular trajectories are the integral curves of the drift $f_0$ ; with each of these curves we associate the adjoint vector $p(t)$ that satisfies the differential equation $\dot{p}(t)=-p\partial_q f_0$,
with initial condition $p(0)\in \{f_i,f_{ij},f_{0i}: \ i,j=1,\ldots,n-1\}^{\bot}$ and $\langle p(0),f_0\rangle=1$ (thanks to (iii), these conditions uniquely define $p(0)$). It is easy to prove
that the pair $(p(t),q(t))$ 
is a normal singular extremal for both Dubins and dodgem car problems, and that it satisfies Assumptions 
\ref{ass lie}--\ref{ass F_0i=0}; in particular $\mathbb{L}_{\wla(t)}=- \mathbb{I}_n$.

\begin{remark}
%
We stress that even if in this problem we consider bounded controls, nevertheless we do not need to strengthen the natural optimality conditions. 
Indeed, High Order Goh condition reduces to Goh condition; moreover, when considering the second variation, the linear quadratic problem  
\eqref{eq: secvar}-\eqref{eq: zeta dot}-\eqref{eq: zeta bc1} coincides with \eqref{eq: secvar fin}-\eqref{eq: zeta extended},
 since the tangent space to $N_0$ is in direct sum with the linear span of the controlled vector fields and their sum coincides with the Lie algebra 
of the controlled vector fields.
\end{remark}

We now compute explicitly the second variation. First of all, we compute the space $\mathcal{W}$ of the admissible variations, that is 
 we shall solve the Cauchy problem for $\zeta(t)$, 
\eqref{eq: zeta dot}-\eqref{eq: zeta bc1}.

Since the reference controls are null, the reference flow reduces to $\wh{S}_t = \exp (tf_0)$.
The time derivatives of the pull-back vector fields give  
\begin{gather*}
\dot{g}_t^i(\wh{q}_0) = \exp \left(-tf_0 \right)_*  \Big[f_0 ,f_i\Big] \circ \exp (tf_0) (\wh{q}_0)=
\exp \left(-tf_0 \right)_*  (f_{0i} ) \circ \exp (tf_0) (\wh{q}_0)\\
\ddot{g}_t^i(\wh{q}_0) = \exp \left(-tf_0 \right)_*  \Big[f_0,f_{0i}\Big] \circ \exp (tf_0) (\wh{q}_0)=
\exp \left(-tf_0 \right)_*  (f_{00i}) \circ \exp (tf_0) (\wh{q}_0)=0,
\end{gather*}
then $\dot{g}_t^i(\wh{q}_0)=\dot{g}_0^i(\wh{q}_0)=f_{0i}(\wh{q}_0)$ for every $t\in[0,\wT]$ and every $i=1,\ldots,N-1$, and then
\[
g_t^i(\wh{q}_0) = f_i(\wh{q}_0)+t f_{0i}(\wh{q}_0).
\]
The solution of \eqref{eq: zeta dot}-\eqref{eq: zeta bc1} is then
\[
\zeta(t) = \sum_{i=1}^{N(N-1)/2}  \epsilon_i f_i(\wq_0) + \sum_{i=1}^{N-1} \int_0^t w_i(s) \; ds\  f_{0i}(\wh{q}_0).
\]
From the boundary condition $\zeta(\wh T)=0$ and from (v) we get that the admissible variations $(\boldsymbol{\epsilon},\boldsymbol{w}(\cdot)) \in \mathcal{W}$
satisfy the constraints
\begin{equation} \label{zeta bc Lie}
\begin{cases}
\int_0^{\wh T} w_i(t) \; dt =  0 & i=1,\ldots,N-1\\
\epsilon_j=0 & j=1,\ldots,\frac{N(N-1)}{2},
\end{cases}
\end{equation}
\noindent
then $\zeta(t) = \sum_{i=1}^{N-1} \left( \int_0^t w_i(s) \; ds \right)  f_{0i}(\wh{q}_0)$.
Then the quadratic form \eqref{eq: secvar fin} is given by
\[
J^{\prime \prime}[(\boldsymbol{0},\boldsymbol{w})]^2=\frac{1}{2}\sum_{i=1}^{N-1}  \int_0^{\widehat{T}} w_i(t)^2   \;dt+ 
\sum_{i,j=1}^{N-1} \int_0^{\widehat{T}} w_i(t) \left( \int_0^t w_j(s) \; ds  \right) L_{f_{0j}} L_{f_{0i}} \beta(\widehat{x}_0) \; dt. 
\]
Integrating by parts the second term and thanks to conditions \eqref{zeta bc Lie} we get
\[
J^{\prime \prime}[(\boldsymbol{0},\boldsymbol{w})]^2=\frac{1}{2}\|\boldsymbol{w}\|^2_{L^2} + \frac{1}{2}\sum_{i=1}^{N-1}
\left(\int_0^{\wT} w_i(t) \;dt\right)^2=\frac{1}{2}\|\boldsymbol{w}\|^2_{L^2}.
\]

Therefore the second variation is
coercive.

\begin{remark}
In this paper we considered the second variation associated with the sub-problem with fixed final point, and we proved
that its coercivity is a sufficient condition for the optimality also if the final condition is not fixed. In particular, this implies that sufficient
optimality conditions for Dubins' problem are also sufficient for the optimality of the extremal in dodgem car problem.

We would like to remark that, in the example considered in this section, the extended second variation associated with the 
original boundary conditions is not coercive. 
Indeed, the final constraint $g(\wT) \in N_f$ imposes the constraint $\zeta(\wT)\in \widehat{S}_{\wT*}^{-1}(T_{g(\wT)}N_f)$. In particular, since $N_f$ is an 
integral manifold of the derived sub-algebra $\{[f_i,f_j] : i,j=1,\ldots,N-1\}$, then $T_{g(\wT)}N_f=\{[f_i,f_j](g(\wT)) : i,j=1,\ldots,N-1\}$ and,
by $\mathrm{(vi)}$ and the fact that the reference flow is the flow associated with the drift, it turns out that  
$\widehat{S}_{\wT*}^{-1}(T_{g(\wT)}N_f)=\{[f_i,f_j](g(0)) : i,j=1,\ldots,N-1\}$.
It is easy to prove that any non-zero variation of the form  $(\boldsymbol{\epsilon},\boldsymbol{w}\equiv 0)$ with $\epsilon_i=0$ for $i\leq N-1$ is admissible for
the problem with final constraint $N_f$, but $J^{\prime \prime}[(\boldsymbol{\epsilon},\boldsymbol{0})]^2=0$.

Therefore, $\mathcal{W}$ is the maximal subspace of variations where we can 	
require coercivity of the extended second variation.
\end{remark}


\appendix

\section{Necessity of HOGC} \label{app variations}
This section is devoted to prove that HOGC is a necessary optimality condition. To be more precise we prove
 the following  result.
\begin{theorem}\label{th:var}
Let  $U=\mathbb{R}^m$ and let $(\wxi,\wu ,\wT)$ be an optimal triple for the problem \eqref{eq: min T}-\eqref{eq: contr sys}.
Then there exists an adjoint covector $\wla :[0,\wT]\to T^*M $, such that 
\[
\langle \wla(t), f (\wxi(t))\rangle =0 \qquad \forall \, f\in \Lie{\mathfrak{f}},\  t\in[0,\wT] .
\]
\end{theorem}
This theorem is already known when the reference control is smooth (see \cite{BianchSte88}); here we remove the 
smoothness hypothesis. The proof follows the outlines of the so called ``higher order maximum principles'' based on
the ``good'' needle-like control variations; in particular we use the results contained in \cite{Bianchini99,Bianchini,Bianchini95}, where most of the other conditions are also analyzed.

A necessary condition for the trajectory $\wxi$ to be time optimal is that $\wxi(t)$ belongs to boundary of the reachable set from $\wq_0 $
at each time $t\in [0,\wT]$ (see for instance \cite{Bianchini95}). Therefore, it is not difficult to see that Theorem \ref{th:var} follows,
if we prove that  $ \mathrm{Lie}_{\wxi(t)}({\mathfrak{f}})$
is a local regular tangent cone to the reachable set at $\wxi(t)$, for every 
$t\in[0,\wT] $ (see \cite{Bianchini99}, \cite[Proposition 3.3]{Bianchini95} and the references therein).
Namely, following \cite{Bianchini99}, it is sufficient to prove that, for every Lebesgue point 
$\bar s\in [0,\wT]$ of the reference control $\wu $, there exists $c>0$ such that $cf(\wxi(\bar s))$ is a g-variations of $(\wxi,\wu)$ at $\bar s$,
 for all $f\in \mathrm{Lie}({\mathfrak{f}})$.

We start by fixing some notations. We denote with $S(t,t_0,q_0,\mathbf{v})$ the solution at time $t$ of the control 
problem \eqref{eq: contr sys} associated
with the control function $\mathbf{v}$, with initial condition at $t_0$ equal to $q_0$. 
Moreover, we use the notation $\widehat{S}_{t,t_0}(q_0)=\widehat{S}(t,t_0,q_0)=
S(t,t_0,q_0,\wu)$, where $\wu$ is the reference control. 

Applying the results in \cite{Bianchini99} and, in particular, putting together Definition 2.1,  Definition 2.3 and Proposition 2.4
therein, it is easy to see that if Lemma \ref{le perturbations} below holds true, then
$cf(\wxi(\bar s))$
is a right g-variation 
of order $2$ of
$(\wxi,\wu)$ at $\bar s$. As a consequence, Lemma \ref{le perturbations} proves Theorem \ref{th:var}. 
\begin{lemma} \label{le perturbations}
Let $\bar s\in [0,\wT]$ be a Lebesgue point for $\wu $. Then there exist positive numbers $c$, $N$, $\bar{\epsilon}$
such that for every $f \in\mathrm{Lie}(\mathfrak{f})$ there exists a family of control maps 
$\{\bnu_{\epsilon} : \epsilon \in [0,\bar{\epsilon}]\} \subset
L^1_{loc}([0,\wT],\mathbb{R}^m)$ with the 
following properties:   
\begin{enumerate}
 \item $\bnu_{\epsilon}(t)=\wu(t)$ outside the interval $[\bar s,\bar s +(N\epsilon)^2]$.
\item There exists a neighborhood $V$ of $\,\wxi(\bar s)$ such that the map 
$$(q,\epsilon)\mapsto 
\widehat{S}_{\bar s,\bar s +(N\epsilon)^2}\circ S(\bar s +(N\epsilon)^2,\bar s,q,\bnu_{\epsilon})$$
is continuous on  $V\times[0,\bar \epsilon]$.
\item The map
 $\epsilon\mapsto 
D_q\,\widehat{S}_{\bar s,\bar s +(N\epsilon)^2}\circ S(\bar s +(N\epsilon)^2,\bar s,q,\bnu_{\epsilon})\big |_{q=\wxi(\bar s)}$
 is continuous.
\item $ \widehat{S}_{\bar s,\bar s +(N\epsilon)^2}\circ S(\bar s +(N\epsilon)^2,\bar s,\wxi(\bar s),\bnu_{\epsilon})
=\wxi(\bar s)+\epsilon cf(\wxi(\bar s)) + o(\epsilon)$.
\end{enumerate}
\end{lemma}

\begin{proof}
Let us consider
the driftless control system 
\begin{equation} \label{driftless}
\dot{\zeta}=\sum_{i=1}^m u_i f_i\circ \zeta,
\end{equation}
denoting its solutions at time $t$, relative to the control $\mathbf{u}$, and with initial condition $\zeta(t_0)=\zeta_0$, as 
$\widetilde{S}(t,t_0,\zeta_0,\mathbf{u})$.

Set $\ba=\widehat{\xi}(\bar s)$. We perform the proof in an adapted coordinate frame centered
at $\ba$,  analogous to the frame \eqref{frame} described in Section \ref{sec secvar}. In this frame
$\ba=\bzero $, $T_{\ba}\mathcal{I}_{\ba}=\R^R$ and $\mathcal{I}_{\ba} $ is a neighborhood of $\bzero$
in $\R^R$, that we call $\mathcal{I}_{\ba} $ in what follows.
 
If $\bar \bt=(\bar t_1,\ldots,\bar t_R)$ is sufficiently small, there exist a choice of $R$ vector fields  
$\{f_{i_1},\ldots,f_{i_R}\}\in \{f_{1},\ldots,f_{m}\}$ such that the map
\begin{equation} \label{map exp}
\bt=( t_1,\ldots, t_{R})\in\mathbb{R}^R\mapsto \exp\left(t_R f_{i_R}\right) \circ \cdots 
\circ \exp\left(t_1 f_{i_1}\right) (\ba)\in \mathcal{I}_{\ba}
\end{equation}
has maximal rank at $\bt=\bar \bt$  (see \cite[Theorem 1]{krenerchow}). This implies that there exist a $\delta>0$ 
and a small neighborhood of 
$\bb= \exp\left(t_R f_{i_R}\right) \circ \cdots 
\circ \exp\left(t_1 f_{i_1}\right) (\ba)$ in 
 $\mathcal{I}_{\bb}$ such that the map \eqref{map exp} is invertible between the ball of radius $\delta$ centered 
 in $\bar \bt$, denoted as $B_{\delta}(\bar \bt)$, 
 and the neighborhood of $\bb$. 
 
For $\bt\in \mathbb{R}^R$, let us rewrite the map \eqref{map exp} as $\widetilde{S}(1,0,\ba,\mathbf{u}^{\bt})$, for the 
piecewise-constant control $\mathbf{u}^{\bt}\in L^{\infty}([0,1],\mathbb{R}^m)$, defined, for $j=1,\ldots,m $,
by 
\[ u^{\bt}_j\colon
s\in\Big[\frac{k-1}{R},\frac{k}{R}\Big]\mapsto
\begin{cases}
R t_k & \mbox{ if } j=i_k\\
0 & \mbox{ if } j\neq i_k
\end{cases} \qquad k=1,\ldots,R.
\]
It is clear that 
$\|\mathbf{u}^{\boldsymbol{t}}\|_{L^{\infty}}$ is uniformly bounded for $\bt \in B_{\delta}(\bar \bt)$.
Let $\mathbf{u_0} \in L^{\infty}([0,1],\mathbb{R}^m)$ be the control map which satisfies
\[
\widetilde{S}(1,0,q,\mathbf{u_0)})= \exp\left(-\bar t_1 f_{i_1}\right) \circ \cdots 
\circ \exp\left(-\bar t_R f_{i_R}\right) (q),
\]
so that $\widetilde{S}(1,0,\bb,\mathbf{u_0)})=\ba$
and  $\|\mathbf{u_0}\|_{L^{\infty}}=\|\mathbf{u}^{\bar{\boldsymbol{t}}}\|_{L^{\infty}}$.

For $\bt\in  B_{\delta} (\bar \bt)$,
let us consider the control map $\bnu_{\bt} \in L^{\infty}([0,2],\mathbb{R}^m)$ defined as
\[
\bnu_{\bt} \colon s\mapsto\begin{cases}
                            \mathbf{u}^{\boldsymbol{t}}(s) & \forall\; s\in[0,1] \\ 
                            \mathbf{u_0}(s-1) & \forall\; s\in (1,2] .
                            \end{cases}
\]
By definition there exists an $M>0$ such that $\|\bnu_{\bt}\|_{L^1([0,2])}\leq M$ for every $\bt\in B_{\delta}(\bar{\bt}) $; moreover,
possibly restricting $\delta$, 
the map $\bt\in B_{\delta}(\bar{\bt})\mapsto \widetilde{S}(2,0,\ba,\bnu_{\bt})$ 
is well defined and covers a compact neighborhood $U$ of $\ba$ in $\mathcal{I}_{\ba}$ contained in the local coordinate chart.

For $\epsilon>0$, we define the control variation $\bnu_{\bt,\epsilon} \in L^1([0,2],\mathbb{R}^m)$ as
\begin{equation}\label{variation L1}
\bnu_{\bt,\epsilon}(s)=
\begin{cases}
\epsilon^{-1} \bnu_{\bt}(s\epsilon^{-2}) & \forall \; s\in[0,2\epsilon^{2}]\\
0 & \forall \; s\in (2 \epsilon^{2},2], 
\end{cases} 
\end{equation}
and the control function
\begin{equation}\label{variation ctr}
\wbnu_{\bt,\epsilon}(s)=
\begin{cases}
\wu(\bar s+s) + \bnu_{\bt,\epsilon}(s) & \forall \; s\in[0,2\epsilon^{2}]\\
\wu(\bar s+s) & \forall \; s\in  [-\bar s,0) \cup (2 \epsilon^2,\wT -\bar s]. 
\end{cases} 
\end{equation}
It is easy to see that $\|\bnu_{\bt,\epsilon}\|_{L^1}=\epsilon\|\bnu_{\bt}\|_{L^1}\leq M\epsilon$ and that 
\begin{equation}\label{eq:senzadrift}
\widetilde{S}(2\epsilon^{2},0,\ba,\bnu_{\bt,\epsilon})=\epsilon \widetilde{S}(2,0,\ba,\bnu_{\bt}) .
\end{equation}
It is clear that $\wbnu_{\bt,\epsilon}$ satisfies property \emph{(1)} 
of Lemma \ref{le perturbations}. To prove the other properties, consider 
the pull-back system $\eta(s)=\widehat{S}_{\bar s+s,\bar s }^{-1}\circ S(\bar s+s, \bar s, q,\wbnu_{\bt,\epsilon})$, 
which is solution of the following Cauchy problem
\begin{equation}\label{eq:pullb}
 \begin{cases}
 \dot{\eta}(s)=\sum_{i=1}^m \nu_{\bt,\epsilon}^i(s) \Big(\widehat{S}_{\bar s+ s,\bar s}^{-1}\Big)_*f_i\circ \widehat{S}_{\bar s+ s,\bar s}(\eta(s))\\
\eta(0)=q,
\end{cases}
\end{equation}
where $\nu^i_{\bt,\epsilon}$ denotes the $i$-component of $\bnu_{\bt,\epsilon} $.
Since $\|\bnu_{\bt,\epsilon}\|_{L^1}\leq M\epsilon$, possibly restricting $\delta$, there exist 
$\bar\epsilon \in (0,1) $
and a neighborhood $V$ of $\ba$ such that $\eta (2\epsilon^{2})$ belongs to $U$, for all $q\in V$,
$\bt\in  B_{\delta}(\bar{\bt})$ and 
$\epsilon\in [0,\bar\epsilon]$.

Fix $\bt\in B_{\delta}(\bar{\bt})$. 
It is not difficult to verify that 
the map $\epsilon\in [0,\bar\epsilon ]\mapsto \bnu_{\bt,\epsilon}$ is strongly continuous in $L^1([0,2],\mathbb{R}^m)$; therefore,
\emph{(2)} and \emph{(3)} of Lemma \ref{le perturbations} are consequences of the properties of
system \eqref{eq:pullb}, see \cite{McShane83}.

Finally, to verify property \emph{(4)}, we consider the system \eqref{eq:pullb} with initial condition
$\eta(0)=\ba $ and the system  \eqref{driftless} with the same initial condition and control map
$\mathbf{u}= \bnu_{\bt,\epsilon} $.
We get
\begin{align*}
\frac{d}{ds}\left| \eta(s)-\zeta(s) \right| &\leq \sum_{i=1}^m  |\nu^i_{\bt,\epsilon}(s)| \Big( 
|(\widehat{S}_{\bar s+ s,\bar s}^{-1})_*f_i\circ \widehat{S}(\bar s+s,\bar s,\eta(s))-f_i(\zeta(s))|\Big)\\
&\leq \sum_{i=1}^m  |\nu^i_{\bt,\epsilon}(s)| \Big( |f_i(\eta(s))-f_i(\zeta(s))|+
|(\widehat{S}_{\bar s+ s,\bar s}^{-1})_*f_i\circ \widehat{S}(\bar s+s,\bar s,\eta(s))-f_i(\eta(s))|\Big)
\end{align*}
Possibly restricting $\bar\epsilon $ and $\bar\delta $, $\eta (s) $ and $\zeta (s) $ belong to the compact neighborhood
$U$, therefore there exists a constant $C>0$ such that 
$$ \frac{d}{ds}\left| \eta(s)-\zeta(s) \right|
\leq C  \sum_{i=1}^m  |\nu^i_{\bt,\epsilon}(s)|  |\eta(s)-\zeta(s)| +
   C\epsilon^{-1} s. 
$$
By Gronwall inequality we obtain 
\[
|\eta(2 \epsilon^{2})-\zeta(2\epsilon^{2}))| \leq e^{CM\epsilon}
\int_{0}^{2\epsilon^{2}} C \epsilon^{-1}s\;ds = 4 C e^{CM\epsilon} \epsilon^{3}.
\]
so that
\[
\eta(2 \epsilon^{2})=\zeta(2\epsilon^{2})+o( \epsilon).
\]
Since $\zeta(2\epsilon^{2})=\epsilon \widetilde{S}(2,0,\ba,\bnu_{\bt})$ and 
$\widetilde{S}(2,0,\ba,\bnu_{\bt})$ covers a neighborhood of $\ba$ in $\R^R$, Lemma \ref{le perturbations}
is proved.
\end{proof}

\section{Orthonormal frame bundles on canonical space forms} \label{Lie}

In this section we give more details about the lifting of Dubins' and dodgem car problem on Lie groups. 
For details, we refer to \cite{Jurdjevic,JurdjLie} and references 
therein.

Let $M \in\{\mathbb{R}^n,S^n,\mathbb{H}^n\}$. We recall that the hyperbolic space $\mathbb{H}^n$ is defined as 
$\mathbb{H}^n=\{ \boldsymbol{x}\in \mathbb{R}^{n+1} : -x_0^2 +\sum_{i=1}^n x_i^2=1,\ x_0>0\}$. 
The manifolds $\mathbb{R}^n$ and $S^n$ inherit a natural Riemannian structure from $\mathbb{R}^n$ and $\mathbb{R}^{n+1}$, respectively. As for 
$\mathbb{H}^n$, its Riemannian metric is given by the Lorentzian quadratic form $\langle \boldsymbol{x},\boldsymbol{y}\rangle=-x_0 y_0+\sum_{i=2}^n x_iy_i$.

The Dubins' problem on $M$ can be lifted to a minimum-time problem on the bundle of positive-oriented orthonormal frames on $M$, denoted with $\mathcal{F}_+(M)$, 
as we show below.

For $M=\mathbb{R}^n$, let us fix some positively oriented orthonormal frame $\{e_1,\ldots,e_n\}$ attached at the point $\boldsymbol{q}=0$ in $\mathbb{R}^n$. 
Given a point $\tilde{\boldsymbol{q}}\in \mathbb{R}^n$ and a positively oriented orthonormal frame $\{\boldsymbol{v}_1,\ldots,\boldsymbol{v}_n\}$ attached at 
$\tilde{\boldsymbol{q}}$, we can associate to them a pair
$(\boldsymbol{x},R) \in \mathbb{R}^n \times$ SO$(\mathbb{R}^n)$, where $\boldsymbol{x}$ denotes the coordinate representation of $\tilde{\boldsymbol{q}}$ with respect
to the basis $\{e_1,\ldots,e_n\}$, and $\boldsymbol{v}_i=Re_i$ for every $i=1,\ldots,n$. 
In other words, the bundle of positively oriented orthonormal frames can be identified 
with the orbit through $(0,\{e_1,\ldots,e_n\})$ of  
the semi-direct product $G=\mathbb{R}^n\rtimes$ SO$(\mathbb{R}^n)$, that is the group of pairs $(\boldsymbol{x},R) \in \mathbb{R}^n\times$ SO$(\mathbb{R}^n)$ 
equipped with the 
operation $(\boldsymbol{x},R)\cdot(\boldsymbol{y},S)=(\boldsymbol{x}+R\boldsymbol{y},RS)$. 
This construction provides a coordinate 
system on $\mathcal{F}_+(\mathbb{R}^n)$. Moreover, every element $(\boldsymbol{x},R)\in G$ can be represented by the following matrix $g\in $ GL$_{n+1}(\mathbb{R})$ 
\[
g=\begin{pmatrix}
   1 & 0 & \ldots & 0\\
   x_1 & & & \\
 \vdots & & R & \\
 x_n & & & 
  \end{pmatrix}.
\]

As the manifolds  $S^n$ and $\mathbb{H}^n$ are embedded in $\mathbb{R}^{n+1}$, we can repeat the same construction and find some group $G$ such that all the elements of
$\mathcal{F}_+(M)$ are given by the orbit of $G$ through some fixed orthonormal frame $\{\mathfrak{e}_1,\ldots,\mathfrak{e}_{n+1}\}$
of $\mathbb{R}^{n+1}$ centered at some fixed point $\boldsymbol{x}_0$.

Indeed, 
every point $\boldsymbol{q}\in S^n$ can be represented with respect to the canonical basis $\{\mathfrak{e}_1,\ldots,\mathfrak{e}_{n+1}\}$
by a unit vector 
$\boldsymbol{x} \in \mathbb{R}^{n+1}$. The tangent space 
to $S^n$ at $\boldsymbol{q}$ is given by the span of $n$ unit vectors $(\boldsymbol{v}_1,\ldots,\boldsymbol{v}_n)\in \mathbb{R}^{n+1}$ orthogonal to
$\boldsymbol{x}$. A choice of these unit vectors determines an orthonormal frame on the tangent space.
Therefore, 
the bundle $\mathcal{F}_+(S^n)$ can be regarded as the orbit of SO$(n+1)$ applied to the 
standard orthonormal frame $\{\mathfrak{e}_1,\ldots,\mathfrak{e}_{n+1}\}$ of $\mathbb{R}^{n+1}$, in the following way: to a frame $\{\boldsymbol{v}_1,\ldots,\boldsymbol{v}_n\}$ attached at a point 
$\boldsymbol{q}\in S^n$ there 
corresponds the 
matrix $g\in$ SO$(n+1)$ such that the coordinates of $\boldsymbol{q}$ are given by $\boldsymbol{x}=g\mathfrak{e}_1$ and 
$\boldsymbol{v}_i=g\mathfrak{e}_{i+1},\ i=1,\ldots,n$, that is
\[
g=\begin{pmatrix}
   x^1 & v_1^1 & \ldots & v_n^1\\
 \vdots &\vdots &  & \\
 x^{n+1} & v_1^{n+1}& \ldots &v_n^{n+1}  
  \end{pmatrix}
\]
(here $x^j$ and $v_i^j$ denote respectively the j-th component of the vectors $\boldsymbol{x}$ and $\boldsymbol{v}_i$).

For what concerns the hyperboloid $\mathbb{H}^n$, we consider the \emph{Lorentz group} SO(1,$n$), defined as the group of transformation 
that preserve the $(n+1)$-dimensional
matrix  
\[
\mathbb{I}(1,n)=\begin{pmatrix}
-1&  0 & \ldots & 0\\
0 & &\mathbb{I}_n &  
\end{pmatrix},
\]
where $\mathbb{I}_n$ is the $n$-dimensional identity matrix.
It can be proved as above that the bundle $\mathcal{F}_+(\mathbb{H}^n)$ can be identified with
the connected component SO$_0$(1,$n$) of SO(1,$n$) that contains the group identity.

For any $M\in \{\mathbb{R}^n,S^n,\mathbb{H}^n\}$, the tangent vector fields in $\mathcal{F}_+(M)$ are identified with the left-invariant vector fields on the isometry 
group of $M$, which is respectively $\mathbb{R}^n\rtimes$ SO$(\mathbb{R}^n)$ for $\mathbb{R}^n$, SO$(n+1)/$SO$(n)$ for $S^n$ and SO(1,$n$) for $\mathbb{H}^n$.
Therefore, we can prove that the tangent vectors are of the form
\begin{equation} \label{g dot eps 2}
g 
\begin{pmatrix}
0 &-\epsilon a_1 & \ldots & -\epsilon a_n\\
a_1& & & \\
 \vdots && U&\\
a_n && & 
\end{pmatrix},
\end{equation}
where $\epsilon=0$ for $M=\mathbb{R}^n$, $\epsilon=1$ for $M=S^n$ and $\epsilon=-1$ for $M=\mathbb{H}^n$, $U$ is an anti-symmetric matrix of dimension $n$,
and $g$ is the element of $\mathcal{F}_+(M)$ where the vector is attached.

Let $\gamma$ be a curve in $M$. When lifting $\gamma$ to a curve of orthonormal frames, that is to a curve $g(t)\in G$ such that its projection on $M$ coincides with $\gamma$,
it is possible to choose the lifted curve in such a way that the first 
element $\boldsymbol{v}_1(t)$ of the frame attached at $\gamma(t)$ is equal to $\dot{\gamma}(t)$. In particular, this sets $a_1=1$ and $a_j=0$ for $j\geq 2$ in equation 
\eqref{g dot eps 2}.
This kind of lifting is called \emph{Darboux frame}.

There is still a freedom of choice of the form of the matrix $U$ in equations \eqref{g dot eps 2}. Systems with the form \eqref{g dot eps} are called 
\emph{Serret-Frenet curves} (see \cite{Jurdjevic,JurdjLie} and references therein).

\bibliography{2ndvarbibliography}{}
\bibliographystyle{plain}

\end{document}